\documentclass[10pt,a4paper]{article}

\usepackage[english]{babel}
\usepackage{amsmath,amssymb,amsthm,authblk}
\usepackage{pstricks-add}

\renewcommand\labelenumi{(\roman{enumi})}
\renewcommand\theenumi\labelenumi
\renewcommand\labelenumii{\alph{enumii}.}
\renewcommand\theenumii\labelenumii

\usepackage[colorlinks=true,
	linkcolor=black,urlcolor=black,citecolor=black,
]{hyperref}

\usepackage{color}
\newtheorem{nummer}{ }
\newtheorem*{theorem*}{Theorem}
\newtheorem{thm}[nummer]{Theorem}
\newtheorem{prop}[nummer]{Proposition}
\newtheorem{lem}[nummer]{Lemma}
\newtheorem{facts}[nummer]{Facts}
\newtheorem{cor}[nummer]{Corollary}
\theoremstyle{definition}
\newtheorem{rem}[nummer]{Remark}
\newtheorem{defi}[nummer]{Definition}
\newtheorem*{exa}{Example}

\usepackage{tikz}

\DeclareRobustCommand*{\sagelogo}{%
  \begin{tikzpicture}[line width=.2ex,line cap=round,rounded corners=.01ex,baseline=-.2ex]
    \draw(0,0) -- (.75em,0) 
      -- (.75em,.7ex) -- (.25em,.7ex) 
      -- (.25em,.75\ht\strutbox) -- (.75em,.75\ht\strutbox);
    \draw(2em,0) -- (1.6em,0)
      -- (1.3em,.75\ht\strutbox) -- (.9em,.75\ht\strutbox)
      -- (.9em,0) -- (1.3em,0)
      -- (1.6em,.75\ht\strutbox) -- (2.1em,.75\ht\strutbox)
      -- (2.1em,-\dp\strutbox) -- (1.45em,-\dp\strutbox);
    \draw(3em,0) -- (2.25em,0)
      -- (2.25em,.75\ht\strutbox) -- (2.8em,.75\ht\strutbox)
      -- (2.8em,.7ex) -- (2.35em, .7ex);
  \end{tikzpicture}%
}

\let\phi\varphi
\let\bar\overline

\newcommand*{\ZZ}{\mathbb{Z}}

\newcommand*{\miqplane}{\mathbb{M}}

\DeclareMathOperator{\ord}{ord}
\DeclareMathOperator{\capacfunc}{cap}
\newcommand*{\capac}{\ensuremath\kappa}

\definecolor{carrier}{rgb}{1,0,0}
\definecolor{chain}{rgb}{0,0,1}

\title{Exotic Steiner Chains in Miquelian M\"obius Planes}
\author[1]{Norbert Hungerb\"uler}
\author[2]{Gideon Villiger}
\affil[1]{Department of Mathematics, ETH Z\"urich}
\affil[2]{Institute of Mathematics, University of Z\"urich}

\begin{document}
\maketitle
\begin{abstract}
\parskip=2mm
\parindent=0mm
\noindent
We investigate Steiner chains of circles in finite Miquelian M\"obius planes of odd order
. 
In the Euclidean
plane, two intersecting circles or two circles which are tangent to each other
clearly do not carry a finite Steiner chain. However, in this paper we will show
that such exotic Steiner chains exist in finite Miquelian M\"obius planes of odd order.
We state and prove explicit conditions in terms of the order of the plane and  
the capacitance of the two carrier circles 
$C_1$ and $C_2$ for the existence, length, and number of Steiner chains carried 
by $C_1$ and $C_2$. 


\end{abstract}

\section{Introduction}
When Jakob Steiner was still a pupil in Yverdon's Pestalozzi school, he found his famous
theorem in circle geometry\footnote{``Found on Saturday Dec.~10th, 1814, after 
$3 + 3 + 4$ hours of efforts, at 1 o'clock in the night." From Steiner's notes during his first month in Yverdon.}:
\begin{theorem*}[Steiner's porism]
Let $C_1,C_2$ be two disjoint M\"obius circles (circles or straight lines) in the Euclidean 
plane. Consider a sequence of different M\"obius circles $M_1,M_2,\ldots, M_k$ which
are tangent to both $C_1$ and $C_2$. Moreover, let $M_i$ and $M_{i+1}$ be tangent
for $i=1,\ldots, k-1$. Then the following is true: If $M_1$ and $M_k$ are tangent, then
there are infinitely many such chains: Every point of $C_1$ and $C_2$ belongs to a circle of such
a chain. And every chain of consecutive tangent circles closes after exactly $k$ steps.
\end{theorem*}
\begin{figure}[h!]
\begin{center}
\psset{unit=9.0,algebraic=true,dimen=middle,dotstyle=o,dotsize=5pt 0,linewidth=0.6pt,arrowsize=3pt 2,arrowinset=0.25}
\begin{pspicture*}(1.2630081621235665,-0.25244579963318936)(1.7706903111645296,0.2539108085486618)
\rput[tl](1.3,.19){\small{$\red{C_2}$}}
\rput[tl](1.62,0.016){\small{$\red{C_1}$}}
\pscircle[linecolor=carrier,linewidth=1pt](1.5214157397027037,0.){0.2392921301486482}
\pscircle[linecolor=carrier,linewidth=1pt](1.6397114703205506,0.){0.02205697204739704}
\pscircle[linecolor=chain](1.564945776296287,0.11568675374080538){0.1156867537408055}
\pscircle[linecolor=chain](1.4498890539136045,0.){0.1677654443595491}
\pscircle[linecolor=chain](1.564945776296287,-0.11568675374080538){0.1156867537408055}
\pscircle[linecolor=chain](1.669967066011676,-0.08498222164502009){0.06815045040967235}
\pscircle[linecolor=chain](1.7072841808675658,-0.028446832540566735){0.051259421018968104}
\pscircle[linecolor=chain](1.7072841808675658,0.028446832540566735){0.051259421018968104}
\pscircle[linecolor=chain](1.669967066011676,0.08498222164502009){0.06815045040967235}
\end{pspicture*}
\psset{unit=7.0cm,algebraic=true,dimen=middle,dotstyle=o,dotsize=5pt 0,linewidth=0.6pt,arrowsize=3pt 2,arrowinset=0.25}
\begin{pspicture*}(0.3026803817200881,-0.12300131820504885)(0.9614170378570972,0.5263471670606514)
\rput[tl](.4,.14){\small{$\red{C_2}$}}
\rput[tl](.825,0.29){\small{$\red{C_1}$}}
\pscircle[linecolor=carrier,linewidth=1pt](0.4817817505932379,0.07048567069821618){0.1097438105343636}
\pscircle[linecolor=carrier,linewidth=1pt](0.768923325749653,0.11249507949238614){0.15236914545244162}
\pscircle[linecolor=chain](0.6363757362608823,0.19759302931943953){0.3098827148490222}
\pscircle[linecolor=chain](0.5486707790689813,0.32600070288085553){0.15438129959665747}
\pscircle[linecolor=chain](0.5933012775658847,0.1499805155043452){0.027208864674824647}
\pscircle[linecolor=chain](0.6011805371845318,0.10807405416531134){0.015431893222965894}
\pscircle[linecolor=chain](0.605648953873708,0.07859057709441022){0.014388270557577968}
\pscircle[linecolor=chain](0.6098087925588451,0.04331129816339502){0.021135408584351242}
\pscircle[linecolor=chain](0.6158644217801493,-0.044518998343871505){0.06690339893721997}
\end{pspicture*}
\psset{unit=1.3cm,algebraic=true,dimen=middle,dotstyle=o,dotsize=5pt 0,linewidth=0.6pt,arrowsize=3pt 2,arrowinset=0.25}
\begin{pspicture*}(-0.9259672057956748,-1.865231186361396)(2.680199674355029,1.740935693789325)
\rput[tl](.39,.1){\small{$\red{C_2}$}}
\rput[tl](2.25,0){\small{$\red{C_1}$}}
\pscircle[linecolor=carrier,linewidth=1pt](0.5667271761044993,0.25984680376355396){0.1193201844053403}
\pscircle[linecolor=carrier,linewidth=1pt](1.5940781466744594,0.7308915626913759){0.9627026511386289}
\pscircle[linecolor=chain](0.44252648542481166,-0.8269387405543815){0.9745393270968191}
\pscircle[linecolor=chain](-0.09549700308597153,0.7881894165954275){0.7278437784958822}
\pscircle[linecolor=chain](0.6121223869523671,0.43599096980871976){0.06257949233892807}
\pscircle[linecolor=chain](0.6715261139986257,0.36557509583484676){0.029546497992198106}
\pscircle[linecolor=chain](0.6975184835419721,0.31866127350484413){0.02408660745666557}
\pscircle[linecolor=chain](0.7159447985457202,0.2676996665189162){0.030103931119782774}
\pscircle[linecolor=chain](0.7301304700275384,0.17285338441957607){0.06579732280037028}
\end{pspicture*}
\end{center}
\caption{Left: Steiner chain wrapping around twice. 
Middle and right: carrier circles which are not nested.}\label{fig-steiner-ketten}
\end{figure}
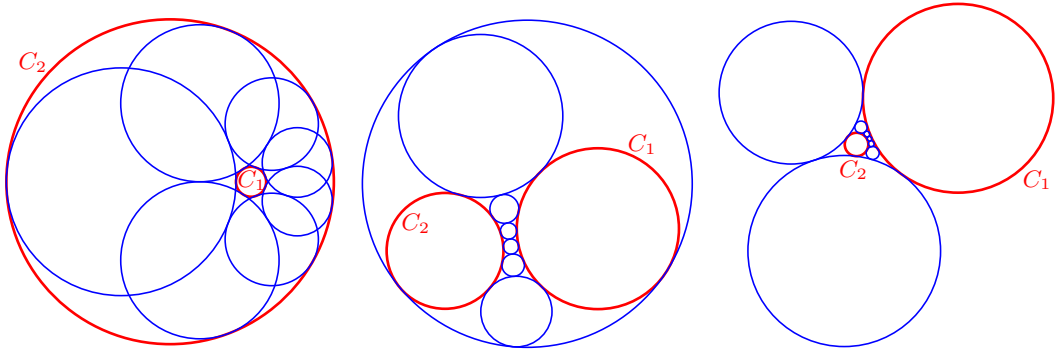
In the sequel Steiner  investigated the geometric properties of such  chains.   
For example,  he  proved  that  the tangent
points of  the circles  $M_1,\ldots,M_k$ lie  on a
circle and that their centers lie on a conic whose foci are  the centers of the
carrier circles $C_1$ and $C_2$.  He also stated the conditions for such a
chain to close after $k$ steps in  terms of the radii and the distance
between the centers of $C_1$ and $C_2$.  The interested reader will find
more information about the classical theory of Steiner chains and
generalizations in~\cite{MR0389515},  
\cite{MR1017034}, \cite{MR990644}, \cite{MR2877262}, or~\cite{MR3193739}.

Throughout this paper, $p$ will denote an odd prime number, $m\ge 1$ a natural number, and $q=p^m$.
It is known that a version of Steiner's porism holds in a finite
Miquelian M\"obius plane $\mathbb M(q)$. However, unlike in the 
Euclidean plane, a pair of circles $C_1, C_2$ in such a finite plane 
may or may not have a common tangent circle. 
If we fix a pair of disjoint circles and choose a point $P$ on one of them,
then the following is true: if $q\equiv -1\mod 4$ and the given pair of 
disjoint circles admits a common tangent circle, then the pair will carry precisely one 
Steiner chain such that $P$ is a point of one of its circles. 
If $q\equiv 1\mod 4$ and if the given pair of disjoint circles 
admits a common tangent circle, then there exists either no Steiner chain, or precisely 
two of them, each with a circle containing $P$. The full
statement is the following theorem (see Section~\ref{sec-preliminaries} for definitions):
\begin{theorem*}[Theorem 5.5 in~\cite{nhkk}]
Let  $C_1$ and $C_2$ be disjoint circles in the Miquelian M\"obius plane $\mathbb M(q)$,
$c:=\operatorname{cap}(C_1,C_2)$ their capacitance, and $P$ an arbitrary point on $C_1$ or $C_2$. Then
$b:=\frac12(c-2+\sqrt{c(c-4)})\in GF(q)\setminus\{0\}$. If $b$ is a nonsquare in $GF(q)$, then
$C_1$ and $C_2$ have no common tangent circles and hence they do not carry a Steiner chain.
If, on the other hand, $b = \mu^2$, for $\mu=\mu_1$ and $\mu=\mu_2=-\mu_1\neq \mu_1\in GF(q)$, then
for each $j\in\{1,2\}$ satisfying the following conditions there is a separate Steiner chain of length $k\ge 3$
carried by $C_1$ and $C_2$ such that $P$ belongs to one of its circles: 
\begin{enumerate}
\item $-\mu_j $ is a non-square in $GF(q)$,
\item $\mu_j $ solves $ \xi^k=1$ for $\xi$ given by 
      \begin{align}
      \xi = \frac{-\mu_j ^2+6\mu_j -1 + 4(\mu_j -1)\sqrt{-\mu_j }}{(1+\mu_j )^2}
      \end{align}
      but $\xi^l \neq 1$ for all $1 \leq l \leq k-1$.
\end{enumerate}
\end{theorem*}
Now, in the Euclidean plane finite Steiner chains cannot exist
if $C_1$ and $C_2$ intersect or are tangent to each other. In the
latter case, the situation corresponds to a Pappus chain (see Figure~\ref{fig-pappus}).
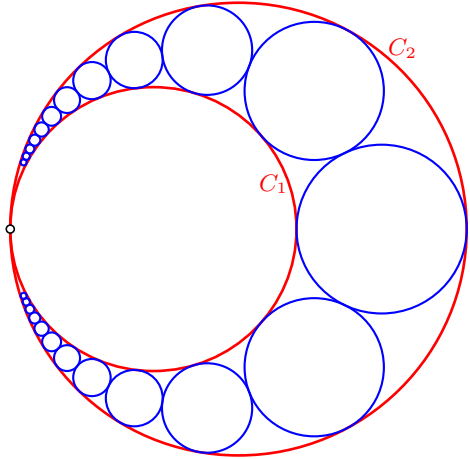
\begin{figure}
\begin{center}
\psset{unit=.7cm,algebraic=true,dimen=middle,dotstyle=o,dotsize=5pt 0,linewidth=.6pt,arrowsize=3pt 2,arrowinset=0.25}
\psset{unit=1.0cm,algebraic=true,dimen=middle,dotstyle=o,dotsize=5pt 0,linewidth=0.8pt,arrowsize=3pt 2,arrowinset=0.25}

\begin{pspicture*}(-3.1,-2.1)(3.1,4.1)
\pscircle[linewidth=0.6pt,linecolor=carrier,linewidth=1pt](0.,1.){3.0014115746490155}
\rput[tl](1.9666283531868813,3.52){\small{$\red{C_2}$}}
\rput[tl](0.27,1.7261914686802853){\small{$\red{C_1}$}}
\pscircle[linecolor=carrier,linewidth=1pt](-1.119109612832595,1.){1.88230196181642}
\pscircle[linecolor=chain](1.8823019618164198,1.){1.1191096128325952}
\pscircle[linecolor=chain](0.9961209996049463,2.8320800751539745){0.9160400375769872}
\pscircle[linecolor=chain](-0.4129549046327818,3.3725934948365843){0.5931483737091462}
\pscircle[linecolor=chain](-1.3708644231038893,3.2418526121722){0.3736421020287006}
\pscircle[linecolor=chain](-1.9273389362870093,2.96900167136558){0.24612520892069814}
\pscircle[linecolor=chain](-2.254900345268568,2.7106406562326235){0.17106406562326193}
\pscircle[linecolor=chain](-2.457601594984016,2.4953775759039156){0.12461479799199293}
\pscircle[linecolor=chain](-2.589715257040065,2.3207708118612898){0.09434077227580545}
\pscircle[linecolor=chain](-2.6798533047956434,2.1789682525537293){0.07368551578460839}
\pscircle[linecolor=chain](-2.7437806677471483,2.062656521892548){0.05903647343847612}
\pscircle[linecolor=chain](-2.790617735647247,1.9660739099103695){0.048303695495519196}
\pscircle[linecolor=chain](-2.82588692960564,1.8848776559042864){0.04022171163201156}
\pscircle[linecolor=chain](0.9961209996049463,-0.8320800751539745){0.9160400375769872}
\pscircle[linecolor=chain](-0.4129549046327818,-1.3725934948365843){0.5931483737091462}
\pscircle[linecolor=chain](-1.3708644231038893,-1.2418526121722002){0.3736421020287006}
\pscircle[linecolor=chain](-1.9273389362870093,-0.96900167136558){0.24612520892069814}
\pscircle[linecolor=chain](-2.254900345268568,-0.7106406562326235){0.17106406562326193}
\pscircle[linecolor=chain](-2.457601594984016,-0.4953775759039156){0.12461479799199293}
\pscircle[linecolor=chain](-2.589715257040065,-0.32077081186128975){0.09434077227580545}
\pscircle[linecolor=chain](-2.6798533047956434,-0.17896825255372928){0.07368551578460839}
\pscircle[linecolor=chain](-2.7437806677471483,-0.06265652189254789){0.05903647343847612}
\pscircle[linecolor=chain](-2.790617735647247,0.033926090089630545){0.048303695495519196}
\pscircle[linecolor=chain](-2.82588692960564,0.11512234409571365){0.04022171163201156}

	\pscircle[fillcolor=white,fillstyle=solid,linewidth=0.6pt](-3.0014115746490155,1.){1.5pt}

\end{pspicture*}
\caption{Pappus chain}\label{fig-pappus}
\end{center}
\end{figure}

On the other hand, in a finite M\"obius plane there are only
finitely many circles, and therefore it is conceivable that
a Pappus chain closes after finitely many steps. It is the aim
of this paper to investigate the corresponding questions: 
Do Steiner chains exist if the carrier
circles intersect or are tangent to each other?
The easier case, when the carrier circles are tangent,
will be treated in Section~\ref{sec-tangent}, the more delicate case
of intersecting carrier circles is discussed in Section~\ref{sec-intersecting}. 
Since these chains do not exist
in the classical M\"obius plane, we call them {\em exotic Steiner chains}.
\section{Preliminaries}\label{sec-preliminaries}
A M\"obius plane is a triple $(\mathbb{P},\mathbb{B}, \mathbb I)$
of points $\mathbb{P}$, circles $\mathbb{B}$ and an incidence relation
$\mathbb I$, satisfying three axioms: 
\begin{enumerate}
  \item[(M1)] 
    For  any three elements  $P,Q, R \in  \mathbb{P}$, $P \neq  Q$, $P \neq  R$ and $Q \neq  R$, there
    exists  a  unique element  $C   \in  \mathbb{B}$  with  $P \in C$, $Q \in C$ and $R \in C$.
  \item[(M2)] 
    For  any $C \in  \mathbb{B}$, $P,Q   \in  \mathbb{P}$ with  $P \in C$ and $Q \notin C$, 
    there exists a unique element $D \in  \mathbb{B}$ such that $P \in D$ and $Q \in D$, 
    but for all $R \in \mathbb{P}$ with $R \in C, P \neq R$, we have $R \notin D$.
  \item[(M3)] 
    There are four  elements $P_1, P_2, P_3, P_4 \in \mathbb{P}$ such that for all $C \in \mathbb{B}$, 
    we have $P_i \notin C$ for at least one $i \in \{ 1,2,3,4 \}$.
    Moreover, for all $C \in \mathbb{B}$ there exists a $P \in \mathbb{P}$ with $P \in C$.
\end{enumerate}
A Steiner chain in a M\"obius plane is defined as follows: 
\begin{defi}\label{def:steiner-chain}
	Given two circles $C_1,C_2$, we say that they carry a (proper) \emph{Steiner chain of length} $k\geq 3$, if there exists a sequence (chain) of distinct circles $M_1,\ldots,M_k$ such that
	\begin{enumerate}
		\item each circle $M_i$ is tangent to the next one $M_{i+1}$, where indices are taken cyclically,
		\item each circle in the chain is tangent to $C_1$ and $C_2$, and
		\item no point is contact point of more than two tangent circles.\label{def:itm:contact-points}
	\end{enumerate}
\end{defi}
The condition~\ref{def:itm:contact-points} excludes degenerate Steiner chains
as the one in Figure~\ref{fig-degenerate}.
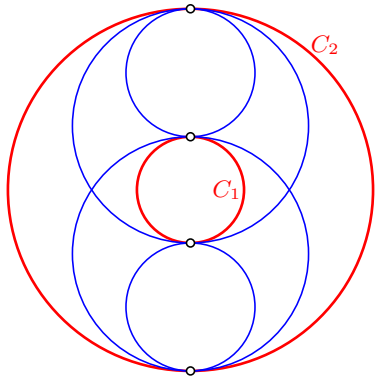
\begin{figure}
\begin{center}

\psset{unit=.8cm,algebraic=true,dimen=middle,dotstyle=o,dotsize=5pt 0,linewidth=0.6pt,arrowsize=3pt 2,arrowinset=0.25}

\begin{pspicture*}(-3.1,-2.15)(3.1,4.15)
	\pscircle[linecolor=carrier,linewidth=1pt](0.,1.){0.880542324717352}
	\pscircle[linecolor=carrier,linewidth=1pt](0.,1.){3.0014115746490155}
	\pscircle[linecolor=chain](0.,-0.06043462496583174){1.9409769496831837}
	\pscircle[linecolor=chain](0.,-0.9409769496831838){1.0604346249658319}
	\pscircle[linecolor=chain](0.,2.9409769496831837){1.0604346249658316}
	\pscircle[linecolor=chain](0.,2.060434624965832){1.9409769496831837}
	\rput[tl](1.9762013337083462,3.56){\small{$\red{C_2}$}}
	\rput[cl](0.6,1.){\small{$\red{C_1}$}}

	\pscircle[fillcolor=white,fillstyle=solid,linewidth=0.6pt](0.,-2.0014115746490155){1.5pt}
	\pscircle[fillcolor=white,fillstyle=solid,linewidth=0.6pt](0.,0.11945767528264795){1.5pt}
	\pscircle[fillcolor=white,fillstyle=solid,linewidth=0.6pt](0.,1.880542324717352){1.5pt}
	\pscircle[fillcolor=white,fillstyle=solid,linewidth=0.6pt](0.,4.0014115746490155){1.5pt}
\end{pspicture*}
\caption{Degenerate Steiner chain}\label{fig-degenerate}
\end{center}
\end{figure}

In order to make this presentation selfcontained, we briefly describe the
construction of a finite Miquel plane, which is based upon the 
Galois field $GF(q)$ and its quadratic extension $GF(q)(\alpha)\cong GF(q^2)$,
where $\alpha$ is a nonsquare in
$GF(q)$. Recall that the conjugation
$$GF(q^2)\to GF(q^2),\quad z\mapsto\bar z:=z^{q}$$
is an automorphism of $GF(q^2)$, whose fixed point set is $GF(q)$
(see, e.g.\ \cite[Theorem 2.21]{lidl:1986:introduction}).
We also define the {\em norm\/} and the {\em trace\/} in the usual way\setlength\arraycolsep{2pt}
$$\begin{array}{lll}
\operatorname{N} :& GF(q^2) \to GF(q),\quad &z \mapsto z\bar z\\[1mm]
\operatorname{Tr}:& GF(q^2) \to GF(q), &z \mapsto z+\bar z.
\end{array}
$$
The finite Miquelian M\"obius plane 
constructed over the pair of finite fields $GF(q)$ and $GF(q^2)$
will be denoted by $\mathbb{M}(q)$, and 
$q$ is called the \emph{order} of $\mathbb{M}(q)$:
The $q^2+1$ points of $\mathbb{M}(q)$ are the elements of $GF(q^2)$
together with a point at infinity, denoted by $\infty$.
There are two different types of circles:
Circles of the first type, are solutions of the equation
$\operatorname{N}(z-c) = r$, i.e.\
\begin{align} \label{circle1}
B^1_{(c,r)}: \ (z-c)(\overline{z}-\overline{c}) = r
\end{align}
for $c \in GF(q^2)$ and $r \in GF(q)\backslash \{0\}$. 
It is easy to see that there are $q+1$ points in $GF(q^2)$ 
on every such circle, and that there are $q^2(q-1)$ circles of the first type.

Circles of the second type are solutions of the equation
$\operatorname{Tr}(\overline{c}z) = r$, i.e.\
  \begin{align} \label{circle2}
    B^2_{(c,r)}: \ \overline{c}z+c\overline{z}=r
  \end{align}
for $c \in GF(q^2)\backslash \{0\}$ and $r \in GF(q)$, together with $\infty$.
Hence, circles of the second type also contain $q+1$ points.
There are $(q^2-1)q$ choices for $c$ and $r$, but scaling with any element of $GF(q)\backslash\{0\}$ leads to the same circle. 
Consequently, there are $q(q+1)$ circles of the second type. 

Now, let $a,b,c,d \in GF(q^2)$ such that $ad-bc \neq 0$. 
The bijective map $\Phi$ defined by
   $$ \Phi: \mathbb{M}(q) \rightarrow \mathbb{M}(q), \ \Phi(z)=\begin{cases}
  \frac{az+b}{cz+d}&\text{if $z\neq\infty$ and $cz+d\neq0$}\\
  \infty   &\text{if $z\neq\infty$ and $cz+d=0$}\\
  \frac ac &\text{if $z=\infty$ and $c\neq 0$}\\
  \infty   &\text{if $z=\infty$ and $c= 0$}\end{cases}$$
is called a \emph{M\"obius transformation} of $\mathbb{M}(q)$.
Every M\"obius transformation is an automorphism of $\mathbb{M}(q)$: It maps circles to circles
and preserves incidence.
Note that a M\"obius transformation operates three times sharply transitive, 
i.e.\ there is a unique M\"obius transformation mapping any three points into any other three given points.
For more background information on finite M\"obius planes, one can refer to \cite{MR1434062}.

The following Lemma states the conditions for the mutual position of two circles.
\begin{lem}\label{prop:inters_circles_first_type}
\begin{enumerate}
\item	Let $B^1_{(c_1,r_1)}$ and $B^1_{(c_2,r_2)}$ be two distinct circles of the first type, and
	\[
		D := (c\bar{c} + r_1 - r_2)^2 - 4c\bar{c}r_1
	\]
	for $c = c_2-c_1$. Then:
	\begin{itemize}
		\item If $D\ne 0$ is a square in $GF(q)$, the circles are disjoint.
		\item If $D = 0$, the circles touch at
		$
			z_0 = \frac{c\bar{c} + r_1-r_2}{2\bar{c}} + c_1 = \frac{1}{2} \left( c_2 + c_1 - \frac{r_2-r_1}{\bar{c}_2-\bar{c}_1} \right).
		$
		\item If $D$ is a nonsquare in $GF(q)$, the circles intersect at
		$
			z_{1,2} = \frac{(c\bar{c} + r_1 - r_2) \pm \sqrt{D}}{2\bar{c}} + c_1.
		$
	\end{itemize}
\item Let $B^1_{(c_1,r_1)}$ and $B^2_{(c_2,r_2)}$ be a circle of the first type and of the second type, respectively, and
	\[
		D := r^2 - 4c_2 \bar{c}_2 r_1
	\]
	for $r = r_2 - c_1\bar{c}_2 - \bar{c}_1c_2$. Then:
	\begin{itemize}
		\item If $D\ne 0$ is a square in $GF(q)$, the circles are disjoint.
		\item If $D = 0$, the circles touch at
		$
			z_0 = \frac{r}{2\bar{c}_2} + c_1
			= \frac{r_2 + c_1\bar{c}_2 - \bar{c}_1 c_2}{2\bar{c}_2}.
		$
		\item If $D$ is a nonsquare in $GF(q)$, the circles intersect at
		$
			z_{1,2} = \frac{r \pm \sqrt{D}}{2\bar{c}_2} + c_1.
		$
	\end{itemize}
\item Let $B^2_{(c_1,r_1)}$ and $B^2_{(c_2,r_2)}$ be two distinct circles of the second type. Then:
	\begin{itemize}
		\item If $c_1 \bar{c}_2 - \bar{c}_1 c_2 = 0$, the circles touch at $\infty$.
		\item If $c_1 \bar{c}_2 - \bar{c}_1 c_2 \ne 0$, the circles intersect at $\infty$ and
		$
			z_0 = \frac{c_1r_2 - c_2r_1}{c_1\bar{c}_2 - \bar{c}_1c_2}.
		$
	\end{itemize}
\end{enumerate}
\end{lem}
The proof is an elementary calculation.

Below we will use M\"obius transformations to bring two general carrier circles 
to a standard position. In order to formulate conditions on the
existence of Steiner chains for intersecting carrier circles in an arbitrary position, we will
need a quantity associated to the two circles,  which remains invariant under 
M\"obius transformations. This invariant is the capacitance, which
was introduced in~\cite{nhkk}:
\begin{defi}\label{def:capacitance}
	The \emph{capacitance} assigns a number in $GF(q)$ to any pair of circles in $\miqplane(q)$. It is defined as
	\begin{align*}
		\capacfunc(B^1_{(c_1,r_1)},B^1_{(c_2,r_2)})
		& := \frac{1}{r_1r_2} \left( r_1 + r_2 - (c_1 - c_2)(\bar{c}_1 - \bar{c}_2) \right)^2, \\
		\capacfunc(B^1_{(c_1,r_1)},B^2_{(c_2,r_2)})
		& := \frac{1}{r_1c_2\bar{c}_2} (c_1\bar{c}_2 + \bar{c}_1c_2 - r_2)^2, \\
		\capacfunc(B^2_{(c_1,r_1)},B^2_{(c_2,r_2}))
		& := \frac{1}{c_1\bar{c}_1c_2\bar{c}_2} (c_1\bar{c}_2 + \bar{c}_1 c_2)^2,
	\intertext{and}
		\capacfunc(B^2_{(c_2,r_2)},B^1_{(c_1,r_1)})
		& := \capacfunc(B^1_{(c_1,r_1)},B^2_{(c_2,r_2)}).
	\end{align*}
\end{defi}

The following Theorem tells us that the capacitance is invariant under M\"obius transformations.

\begin{thm}[Theorem 5.1 in~\cite{nhkk}]
	Let $B, B'\in \miqplane(q)$ be two circles. If $\Phi$ is a M\"obius transformation, then
	\[
		\capacfunc(B, B') = \capacfunc(\Phi(B), \Phi(B')).
	\]
\end{thm}
For the reader's convenience we close this section with a few standard facts about finite fields (see~\cite{lidl:1986:introduction})
which we will tacitly use in the sequel.
\begin{facts}\label{facts}
\begin{itemize}
\item The multiplicative group of a finite field is cyclic. The multiplicative order
of an element in $GF(q)\setminus\{0\}$ divides $q-1$.
\item An element $z\in GF(q)\setminus\{0\}$ is a square in $GF(q)$
if and only if $z^{\frac{q-1}2}=1$.
\item The product of two squares or two nonsquares is a square,
while the product of a nonzero square and a nonsquare is a nonsquare.
\item $-1$ is a square in $GF(q)$ if and only if $q\equiv 1\mod 4$,
or equivalently, if $p\equiv 1\mod 4$ or $m$ even.
\item Any $z\in GF(q)$ is a square in $GF(q^2)$.
\item Let $z\in GF(q)$ be a nonsquare in $GF(q)$ and $\sqrt z$
one of its square roots in $GF(q^2)$. Then $\overline{\sqrt z}=-\sqrt z$.

\end{itemize}
\end{facts}

\section{Exotic Steiner chains in tangent carrier circles}\label{sec-tangent}
\subsection{The standard case}
Let us start with the two circles $B_{-1} := B^2_{(1,-1)}$ and $B_1 := B^2_{(1,1)}$ 
in $\mathbb M(q)$ with equations
\[
	B_{-1}\colon z + \bar{z} = -1\text{\quad and \quad}B_1\colon z + \bar{z} = 1.
\]
These are circles of the second type, and since $p$ is odd, they are different.
Since both equations cannot be satisfied at the same time, the circles are tangent at~$\infty$
(see also Lemma~\ref{prop:inters_circles_first_type}). 

Let $\tau(B_{-1}, B_1)$ denote the set of all common tangent circles of $B_{-1}$ and $B_1$ of the first type.
Observe that circles of the 
second type cannot be part of a  proper Steiner chain carried by $B_{-1}$ and $B_1$, 
because $\infty$ is already used as the contact point of the carrier circles $B_{-1}$ and $B_1$ 
(see Definition~\ref{def:steiner-chain}\ref{def:itm:contact-points}). 
This is why we limit our search to circles of the first type.

According to Lemma~\ref{prop:inters_circles_first_type}, $B^1_{(c,r)}$ is in $\tau(B_{-1}, B_1)$ if and only if
\[
	(c + \bar{c} + 1)^2 = 4r \quad \text{and} \quad
	(c + \bar{c} - 1)^2 = 4r.
\]
This implies $c + \bar{c} = 0$ and $4r = 1$. The condition for $c$ is the equation 
of a circle of the second type, so
there are $q$ circles of the first type in $\tau(B_{-1}, B_1)$. We summarize our findings 
in a Lemma. Notice that $4\ne 0$ because $p$ is odd.

\begin{lem}\label{lem:tangent_circles_b1bm1}
	There are $q$ circles of the first type tangent to both 
	$B_{-1}$ and $B_1$. They are given by 
	$B^1_{(c,r)}$ with $c\in B_0 := B^2_{(1,0)}$ and $r = \frac{1}{4}$.
\end{lem}
As we are trying to construct a chain of circles, let us pick $B^1_{(0,\frac{1}{4})}\in \tau(B_{-1}, B_1)$ as our starting circle. Lemma~\ref{lem:condition_for_chain} tells us under what circumstances such a chain may possibly exist.

\begin{lem}\label{lem:condition_for_chain}
	If $-1$ is a nonsquare in $GF(q)$, then there are exactly 
	two circles $B^1_{(c,r)}\in \tau(B_{-1}, B_1)$ tangent to 
	$B^1_{(0,\frac{1}{4})}$. They are given by $c = \pm \sqrt{-1}$ and $r = \frac{1}{4}$.
	If $-1$ is a square in $GF(q)$, there are no common tangent 
	circles of $B_{-1},B_1$ and $B^1_{(0,\frac{1}{4})}$.
\end{lem}
\begin{proof}
	Let $B^1_{(c,r)}$ be in $\tau(B_{-1}, B_1)$, i.e. $c + \bar{c} = 0$ and $r = \frac{1}{4}$.
	For $B^1_{(c,r)}$ to be tangent to $B^1_{(0,\frac{1}{4})}$ as well, it has to 
	satisfy the  condition from Lemma~\ref{prop:inters_circles_first_type}
	\[
		(c\bar{c} + \frac{1}{4} - \frac{1}{4})^2 = 4c\bar{c}\cdot \frac{1}{4},
	\]
	which is equivalent to
	\[
		c^2(c^2 + 1)=0
	\]
	for $\bar{c} = -c$. As $c\ne 0$ (otherwise $B^1_{(c,r)}$ coincides 
	with $B^1_{(0,\frac{1}{4})}$), it follows that $c^2 = -1$ and thus 
	$c = \pm \sqrt{-1}$, where $\sqrt{-1}$ denotes any square root 
	of $-1$. 
	
	Moreover, the relation $\bar{c} = -c$ makes it clear that $c\notin GF(q)$. 
	Consequently, there only exists a solution if $-1$ is a nonsquare in $GF(q)$. 
\end{proof}
Assume now that $-1$ is a nonsquare in $GF(q)$. As we have seen, in this 
case the two circles $B^1_{(0,\frac{1}{4})}$ and $B^1_{(\sqrt{-1},\frac{1}{4})}$ 
are in $\tau(B_{-1},B_1)$ and are mutually tangent.

At this point we apply the M\"obius transformation  $T: z \mapsto z + \sqrt{-1}$:
Indeed, $T$ leaves $B_{-1}$ and $B_1$ invariant.
On the other hand, $B^1_{(0,\frac{1}{4})}$ is mapped to $B^1_{(\sqrt{-1},\frac{1}{4})}$, 
while $B^1_{(\sqrt{-1},\frac{1}{4})}$ is mapped to $B^1_{(2\sqrt{-1},\frac{1}{4})}$. 
By the properties of M\"obius transformations, both circles are still 
tangent to each other, as well as tangent to $B_{-1}$ and $B_1$.

This induces a Steiner chain: By applying above translation $k$ times, 
we get the $k$-th circle in the chain, and for $k = p$, we are back to our starting circle:
\[
	B^1_{(0, \frac{1}{4})} \to
	B^1_{(\sqrt{-1}, \frac{1}{4})} \to
	B^1_{(2\sqrt{-1}, \frac{1}{4})} \to \cdots \to
	B^1_{(p\sqrt{-1}, \frac{1}{4})} = B^1_{(0, \frac{1}{4})}.
\]
Recall that there are $q=p^m$ circles in $\tau(B_{-1},B_1)$. Hence
there are exactly $p^{m-1}$ Steiner chains of length $p$ each. 
To see this, take an element $c\in B_0$, such that $B^1_{(c,\frac{1}{4})}$ is not in the chain.
The translation $T$ transforms our original chain into a chain starting from 
$B^1_{(c,\frac{1}{4})}$ while leaving the carrier circles $B_{-1}$ and $B_1$ invariant. 
We can repeat this process as long as there are circles left in 
$\tau(B_{-1},B_1)$ that have not been used in a chain.

Therefore we just proved the following
\begin{prop}
	The circles $B^2_{(1,-1)}$ and $B^2_{(1,1)}$ in $\mathbb M(q)$, $q=p^m$, carry a Steiner 
	chain if and only if $q\equiv 3\mod 4$. In this case 
	there are $p^{m-1}$ different Steiner chains, and each chain has length $p$.
\end{prop}
\subsection{The general case}
Let $C_1$ and $C_2$ be two circles in $\mathbb M(q)$ that are tangent at $z_0$.
Choose two  points $z_1,z_2$ on $C_1$ and two points $z'_1,z'_2$ on $B_{-1}$.
Then there is a M\"obius transformation $T_1$ which maps $z_i$ to $z'_i$ and $z_0$ to $\infty$.
Hence $T_1(C_1)=B_{-1}$ and $T_1(C_2)$ is a circle of the second kind tangent to
$B_{-1}$ at $\infty$. 
By Lemma~\ref{prop:inters_circles_first_type}, for any 
circle $B^2_{(c,r)}$ tangent to $B_{-1}$ we have $c\in GF(q)$, 
and therefore $B^2_{(c,r)}$ is given by the equation
$z+\bar{z} = \frac{r}{c}$, with $\frac{r}{c}\in GF(q)\setminus \{-1\}$.
This means that $T_1(C_2)$ has the 
form $z+\bar{z} = r$ for $r\ne -1$. Finally, the M\"obius transformation
$T_2:z\mapsto \lambda(z+\frac12)-\frac12$, with $\lambda=\frac2{r+1}$, maps $B_{-1}$ to itself,
and  $T_1(C_2)$ to $B_1$. Hence $T=T_2\circ T_1$ maps $C_1$ to $B_{-1}$ and
$C_2$ to $B_1$, and an exotic Steiner chain exists for $C_1$ and $C_2$
if and only of there exists one for $B_{-1}$ and $B_1$.
Hence we have the following
\begin{thm}\label{thm:conclusion1}
	Let $C_1$ and $C_2$ be two tangent circles in $\mathbb M(q)$, $q=p^m$.
	If $q\equiv 3\mod 4$, then $C_1$ and
	$C_2$  carry  $p^{m-1}$ Steiner chains, 
	and each chain has length $p$. If $q\equiv 1\mod 4$, $C_1$ and
	$C_2$ do not carry a Seiner chain.
\end{thm}
\section{Exotic Steiner chains for intersecting carrier circles}\label{sec-intersecting}
The case of intersecting carrier circles is particularly more
delicate than the case of tangent carrier circles treated in the previous section.
We start again by a standard situation.
\subsection{The standard case}
Let us start with two different circles of the second type
$B^2_{(\gamma_1,0)}$ and $B^2_{(\gamma_2,0)}$ intersecting 
in $0$ and $\infty$. Then, by Lemma~\ref{prop:inters_circles_first_type},
we have $\gamma_1\bar\gamma_2-\bar\gamma_1\gamma_2\neq 0$ and we 
will prove
\begin{lem}\label{lem:tangent_circles_for_intersecting_carrier_circles}
If $\gamma_1\gamma_2$ is a square in $GF(q^2)$, there are exactly 
$2(q-1)$ circles in $\tau(B^2_{(\gamma_1,0)},B^2_{(\gamma_2,0)})$. 
If $\gamma_1\gamma_2$ is a nonsquare, $B^2_{(\gamma_1,0)}$ and 
$B^2_{(\gamma_2,0)}$ have no common tangent circles.
\end{lem}
\begin{proof}
We start the proof by observing that there are no circles of the second 
type tangent to both $B^2_{(\gamma_1,0)}$ and $B^2_{(\gamma_2,0)}$. By 
Lemma~\ref{prop:inters_circles_first_type}, any such circle $B^2_{(c,r)}$ would satisfy
	\[
		c\bar{\gamma}_1 - \bar{c}\gamma_1 = 0
		\quad \text{and} \quad
		c\bar{\gamma}_2 - \bar{c}\gamma_2 = 0.
	\]
	Above condition leads to
	\[
		c(\gamma_1\bar{\gamma}_2 - \bar{\gamma}_1\gamma_2) = 0,
	\]
	but $c\ne 0$ for a circle of the second type, contradicting $\gamma_1\bar{\gamma}_2 - \bar{\gamma}_1\gamma_2\ne 0$.
	
	For any circle $B^1_{(c,r)}$ of the first type in $\tau(B^2_{(\gamma_1,0)},B^2_{(\gamma_2,0)})$ we have
	\begin{equation}\label{eq:condition_tangent_circle1}
		(c\bar{\gamma}_1 + \bar{c}\gamma_1)^2 = 4\gamma_1 \bar{\gamma}_1 r
	\end{equation}
	and
	\begin{equation}\label{eq:condition_tangent_circle2}
		(c\bar{\gamma}_2 + \bar{c}\gamma_2)^2 = 4\gamma_2 \bar{\gamma}_2 r
	\end{equation}
	as a consequence of Lemma~\ref{prop:inters_circles_first_type}. Notice that $4\gamma_i \bar{\gamma}_i r \ne 0$ because of the way the circles are defined, and since $p$ is odd by assumption. 
Eliminating $r=\frac{(c\bar{\gamma}_2 + \bar{c}\gamma_2)^2}{4\gamma_2\bar{\gamma}_2}$ from~\eqref{eq:condition_tangent_circle1} leads to
\[
	\gamma_2\bar{\gamma}_2 (c^2\bar{\gamma}_1^2 + \bar{c}^2\gamma_1^2)
	= \gamma_1\bar{\gamma}_1 (c^2\bar{\gamma}_2^2 + \bar{c}^2\gamma_2^2)
	\iff c^2\bar{\gamma}_1\bar{\gamma}_2\cdot(\bar{\gamma}_1\gamma_2 - \gamma_1\bar{\gamma}_2)
	= \bar{c}^2\gamma_1\gamma_2(\bar{\gamma}_1\gamma_2-\gamma_1\bar{\gamma}_2).
\]
Since $\gamma_1\bar{\gamma}_2 - \bar{\gamma}_1\gamma_2\ne 0$, this is equivalent to
	\[
		c^2 \bar{\gamma}_1\bar{\gamma}_2=\bar{c}^2 \gamma_1 \gamma_2,
	\]
	and thus any $c$ satisfying \eqref{eq:condition_tangent_circle1} and~\eqref{eq:condition_tangent_circle2} is characterized by the condition
	\[
		c^2 \bar{\gamma}_1\bar{\gamma}_2 \in GF(q),
		\qquad c\ne 0.
	\]
	This means that there must exists an element $\beta\in GF(q)\setminus\{0\}$ 
	such that $c^2 \bar{\gamma}_1\bar{\gamma}_2 = \beta$, or, equivalently,
	\begin{equation} \label{eq:condition_for_c}
		c^2 = \frac{\beta}{\bar{\gamma}_1 \bar{\gamma}_2}
		\quad \text{for } \beta\in GF(q)\setminus \{0\}.
	\end{equation}
	To see under what conditions such a $\beta$ exists, note the following:
	\begin{itemize}
		\item An element $\gamma\in GF(q^2)\setminus\{0\}$ is a square 
		in $GF(q^2)$ if and only if $\bar{\gamma}$ is a square. 
		Indeed, if there exists $a\in GF(q^2)$ with $a^2 = \gamma$, then $\bar{a}^2 = \bar{a^2} = \bar{\gamma}$. Conversely, if $a^2 = \bar{\gamma}$, it follows that $\bar{a}^2 = \gamma$.
		\item It is easy to verify that $\gamma\in GF(q^2)\setminus\{0\}$ is a square if and only if $\frac{1}{\gamma}$ is a square.
		\item If $\beta\in GF(q)\setminus\{0\}$ and $\gamma\in GF(q^2)\setminus\{0\}$, 
		then $\beta\gamma$ is a square if and only if $\gamma$ is a square. 
	\end{itemize}
	To sum up, \eqref{eq:condition_for_c} can be solved for $c$ if and only if $\gamma_1\gamma_2$ is a square in 
	$GF(q^2)$. In this case, $c$ is given by
	\begin{equation}\label{eq:solution_for_c}
		c = \pm \frac{\sqrt{\beta}}{\sqrt{\bar{\gamma}_1\bar{\gamma}_2}}.
	\end{equation}
	There are $q-1$ possible choices for $\beta\in GF(q)\setminus\{0\}$, and thus $2(q-1)$ different values $\pm\sqrt{\beta}$ can attain. Since there is a unique $r$ corresponding to every $c$, 
	there are exactly $2(q-1)$ circles in $\tau(B^2_{(\gamma_1,0)},B^2_{(\gamma_2,0)})$.
\end{proof}

From now on we assume that $\gamma_1\gamma_2$ is a square in $GF(q^2)$, 
i.e.\ $\tau(B^2_{(\gamma_1,0)},B^2_{(\gamma_2,0)})$ is non-empty.
Observe that $\gamma_1\gamma_2$ is a square if and only if both 
$\gamma_1$ and $\gamma_2$ are either squares or nonsquares. 
This also implies (together with what we mentioned in the proof of Lemma~\ref{lem:tangent_circles_for_intersecting_carrier_circles}) that $\gamma_1\gamma_2$ is a square if and only if $\frac{\bar{\gamma}_2}{\bar{\gamma}_1}$ is a square.

At this point, let us define $\gamma$ to be a square root of $\frac{\bar{\gamma}_2}{\bar{\gamma}_1}$:
\[
	\gamma := \sqrt{\frac{\bar{\gamma}_2}{\bar\gamma_1}},
\]
and let us apply the M\"obius transformation $z\mapsto \bar{\gamma}_1\gamma z$ to the carrier circles
\[
	B^2_{(\gamma_1,0)}\colon \bar{\gamma}_1 z + \gamma_1 \bar{z} = 0
	\quad \text{and} \quad
	B^2_{(\gamma_2,0)}\colon \bar{\gamma}_2 z + \gamma_2 \bar{z} = 0.
\]
$B^2_{(\gamma_1,0)}$ is transformed into
\[
	\frac{z}{\gamma} + \frac{\bar{z}}{\bar{\gamma}} = 0
	\iff \bar{\gamma}z + \gamma\bar{z} = 0,
\]
and for $B^2_{(\gamma_2,0)}$ we get
\[
	\frac{\bar{\gamma}_2}{\bar{\gamma}_1} \frac{z}{\gamma}
	+ \frac{\gamma_2}{\gamma_1} \frac{\bar{z}}{\bar{\gamma}} = 0
	\iff \gamma z + \bar{\gamma}\,\bar{z} = 0.
\]
We summarize what we have shown so far: If $\gamma_1\gamma_2$ is a nonsquare, 
no Steiner chain exists. But if $\gamma_1\gamma_2$ is a square, we can always 
transform the circles $B^2_{(\gamma_1,0)}$,$ B^2_{(\gamma_2,0)}$
 into the two symmetric circles 
$B^2_{(\gamma,0)}$ and $B^2_{(\bar{\gamma},0)}$, where $\gamma$ is defined as 
above. 
Notice that the condition $\gamma_1\bar{\gamma}_2 - \bar{\gamma}_1\gamma_2\ne 0$ changes to $\gamma^2\ne \bar{\gamma}^2$.

We will now state an explicit condition for a circle to be 
in $\tau(B^2_{(\gamma,0)},B^2_{(\bar{\gamma},0)})$.
\begin{lem}\label{lem:condition_symmetric_carrier_circles}
There are $2(q-1)$ circles tangent to $B^2_{(\gamma,0)}$ 
and $B^2_{(\bar\gamma,0)}$ with $\gamma^2\ne \bar{\gamma}^2$. 
They are given by $B^1_{(c,r)}$ with $c$ and $r$ satisfying
	\begin{equation}\label{lem:eq:first_circle_condition_1}
		c = \bar{c}, \qquad
		r = c^2 \frac{(\gamma + \bar{\gamma})^2}{4\gamma\bar{\gamma}}
	\end{equation}
	or
	\begin{equation}\label{lem:eq:first_circle_condition_2}
		c = -\bar{c}, \qquad
		r = c^2 \frac{(\gamma - \bar{\gamma})^2}{4\gamma\bar{\gamma}}
	\end{equation}
	for $c\in GF(q^2)\setminus \{0\}$.
\end{lem}

\begin{proof}
By Lemma~\ref{prop:inters_circles_first_type}, the condition 
for a circle $B^1_{(c,r)}$ to be in $\tau(B^2_{(\gamma,0)},B^2_{(\bar{\gamma},0)})$ is
	\begin{equation}\label{eq:condition_c_symmetric_circles}
		\left.
			\begin{aligned}
				(c\bar{\gamma} + \bar{c}\gamma)^2 & = 4\gamma\bar{\gamma}r
				\\
				(c\gamma + \bar{c}\,\bar{\gamma})^2 & = 4\gamma\bar{\gamma}r
			\end{aligned}
		\ \right\}
	\end{equation}
	We subtract the second equation in~\eqref{eq:condition_c_symmetric_circles} from the first and get
	\[
		(c^2 - \bar{c}^2)(\bar{\gamma}^2 - \gamma^2) = 0.
	\]
	Since $\gamma^2\ne \bar{\gamma}^2$, this implies
	\[
		c^2 - \bar{c}^2 = (c - \bar{c})(c + \bar{c}) = 0.
	\]
	Plugging in the respective values $\bar{c} = c$ and $\bar{c} = -c$ in~\eqref{eq:condition_c_symmetric_circles} 
	yields the $r$-values specified in the lemma. We also see that $c$ is nonzero, as $c=0$ would lead to $r = 0$.
\end{proof}
We established in Lemma~\ref{lem:condition_symmetric_carrier_circles} that 
the center $c_1$ of any circle $B^1_{(c_1,r_1)}$ tangent to both carrier
 circles is either on the circle $z-\bar{z} = 0$ (i.e.\ $c\in GF(q)$) 
 or on the circle $z+\bar{z} = 0$. Accordingly, we subsequently 
 investigate what the conditions are for a second 
 circle $B^1_{(c_2,r_2)}\in \tau(B^2_{(\gamma,0)},B^2_{(\bar{\gamma},0)})$ 
 to be tangent to $B^2_{(c_1,r_1)}$ if
\begin{itemize}
	\item both $c_1$ and $c_2$ are on $z-\bar{z} = 0$ (see Lemma~\ref{lem:second_circle_1}),
	\item both $c_1$ and $c_2$ are on $z + \bar{z} = 0$ (see Lemma~\ref{lem:second_circle_2}), and
	\item $c_1$ and $c_2$ are not on the same line (see Lemma~\ref{lem:second_circle_3}).
\end{itemize}
\begin{lem}\label{lem:second_circle_1}
	Let $B^1_{(c_1,r_1)},B^1_{(c_2,r_2)}\in \tau(B^2_{(\gamma,0)},B^2_{(\bar{\gamma},0)})$ with
	\[
		c_1 = \bar{c}_1
		\quad \text{and} \quad
		c_2 = \bar{c}_2.
	\]
	The circles $B^1_{(c_1,r_1)}$ and $B^1_{(c_2,r_2)}$ are tangent if and only if 
	$\gamma\bar{\gamma}$ is a square in $GF(q)$ and
	\begin{equation}\label{lem:eq:second_circle_1}
	c_2 = c_1\cdot \frac{
		2\sqrt{\gamma\bar{\gamma}} \pm (\gamma + \bar{\gamma})}{
		2\sqrt{\gamma\bar{\gamma}} \mp (\gamma + \bar{\gamma})}.
	\end{equation}
\end{lem}
\begin{proof}
Recall that both circles $B^1_{(c_1,r_1)}$ and $B^1_{(c_2,r_2)}$ satisfy 
equation~\eqref{lem:eq:first_circle_condition_1} from 
Lemma~\ref{lem:condition_symmetric_carrier_circles}, namely:
	\begin{equation}\label{proof:eq:local0}
		c_i = \bar{c}_i, \quad
		r_i = c_i^2\frac{(\gamma + \bar{\gamma})^2}{
			4\gamma\bar{\gamma}}, \qquad
		c_i \ne 0,
		\quad i = 1,2.
	\end{equation}
	Moreover, because they are mutually tangent, we also have
	\begin{equation}\label{proof:eq:local1}
		(c\bar{c} + r_1 - r_2)^2 = 4c\bar{c}r_1
		\quad \text{for }
		c := c_2 - c_1
	\end{equation}
	by Lemma~\ref{prop:inters_circles_first_type}. Notice that $c\in GF(q)$, and therefore $c\bar{c} = c^2$.
	Let us write $r_2$ as
	\[
		r_2 = \frac{c_2^2}{c_1^2}r_1 = \left( \frac{c}{c_1} + 1\right)^2 r_1
	\]
	and apply it to equation~\eqref{proof:eq:local1}:
	\[
		\left(
			\left(
				\frac{c}{c_1} + 1
			\right)^2 r_1 - r_1 - c^2
		\right)^2
		= 4c^2 r_1
		\iff \left(
			\left(
				\frac{c^2}{c_1^2} + \frac{2c}{c_1}
			\right) r_1 - c^2
		\right)^2
		= 4c^2 r_1.
	\]
	Dividing both sides by $c^2$, which is nonzero because $B^1_{(c_1,r_1)}$ and $B^1_{(c_2,r_2)}$ are different, yields
	\begin{equation}\label{proof:eq:local2}
		\left(
			c\,\frac{r_1 - c_1^2}{c_1^2} + \frac{2r_1}{c_1}
		\right)^2 = 4r_1.
	\end{equation}
	Notice that
	$c\ \frac{r_1 - c_1^2}{c_1^2} + \frac{2r_1}{c_1} \in GF(q)$, since $c,r_1,c_1\in GF(q)$. 
	Consequently, \eqref{proof:eq:local2} only has a solution if $r_1$ is a square in $GF(q)$. 
	A look at equation~\eqref{proof:eq:local0} makes it clear that $r_1$ is a square in $GF(q)$ 
	if and only if $\gamma\bar{\gamma}$ is a square in $GF(q)$.
	In that case we can write equation~\eqref{proof:eq:local2} as
	\begin{equation}\label{proof:eq:local3}
		c\ \frac{r_1 - c_1^2}{c_1^2} = \pm 2\sqrt{r_1} - \frac{2r_1}{c_1}.
	\end{equation}
	At this point we observe that $r_1 - c_1^2\ne 0$, i.e.\ $\frac{(\gamma + \bar{\gamma})^2}{4\gamma\bar{\gamma}}\ne 1$. In fact,
	\[
		(\gamma+\bar{\gamma})^2 = 4\gamma\bar{\gamma}
		\iff \gamma^2 - 2\gamma\bar{\gamma} + \bar{\gamma}^2 = 0
		\iff (\gamma - \bar{\gamma})^2 = 0
		\iff \gamma = \bar{\gamma},
	\]
	but as we mentioned earlier, $\gamma^2 \ne \bar{\gamma}^2$. We can therefore rearrange~\eqref{proof:eq:local3} by solving for $c$:
	\[
		c = \frac{c_1^2}{r_1-c_1^2}
		\left(\pm 2\sqrt{r_1} - \frac{2r_1}{c_1}\right)
		= \frac{2 c_1\sqrt{r_1} (\pm c_1-\sqrt{r_1})}{
			(c_1 - \sqrt{r_1})(-c_1 - \sqrt{r_1})}.
	\]
	We use that $c_2 = c+c_1$ and get
	\[
		c_2 = \frac{2 c_1 \sqrt{r_1} + c_1(\mp c_1 - \sqrt{r_1})}{
			\mp c_1 - \sqrt{r_1}}
		= c_1\frac{\sqrt{r_1}\mp c_1}{-\sqrt{r_1}\mp c_1}
		= c_1\frac{c_1\mp\sqrt{r_1}}{c_1\pm\sqrt{r_1}}.
	\]
	Finally, substituting $r_1$ gives us
	\[
		c_2 = c_1 \frac{c_1 \mp 
			c_1\frac{\gamma + \bar{\gamma}}{2\sqrt{\gamma\bar{\gamma}}}
		}{ c_1 \pm
			c_1\frac{\gamma + \bar{\gamma}}{2\sqrt{\gamma\bar{\gamma}}}
		}
		= c_1 \frac{2\sqrt{\gamma\bar{\gamma}} \mp (\gamma + \bar{\gamma})}{
			2\sqrt{\gamma\bar{\gamma}} \pm (\gamma + \bar{\gamma})}.
	\]
\end{proof}
\begin{lem}\label{lem:second_circle_2}
	Let $B^1_{(c_1,r_1)},B^1_{(c_2,r_2)}\in \tau(B^2_{(\gamma,0)},B^2_{(\bar{\gamma},0)})$ with
	\[
	c_1 = -\bar{c}_1
	\quad \text{and} \quad
	c_2 = -\bar{c}_2.
	\]
	The circles $B^1_{(c_1,r_1)}$ and $B^1_{(c_2,r_2)}$ are tangent if and 
	only if 
	$-\gamma\bar{\gamma}$ is a nonsquare in $GF(q)$ and
	\begin{equation}\label{lem:eq:second_circle_2}
		c_2 = c_1\cdot \frac{
			2\sqrt{-\gamma\bar{\gamma}} \pm (\gamma - \bar{\gamma})}{
			2\sqrt{-\gamma\bar{\gamma}} \mp (\gamma - \bar{\gamma})}.
	\end{equation}
\end{lem}
\begin{proof}
Recall that both $B^1_{(c_1,r_1)}$ and $B^1_{(c_2,r_2)}$ must satisfy 
equation~\eqref{lem:eq:first_circle_condition_2} from 
Lemma~\ref{lem:condition_symmetric_carrier_circles}:
		\[
			c_i = -\bar{c}_i, \quad
			r_i = c_i^2\ \frac{(\gamma - \bar{\gamma})^2}{
				4\gamma\bar{\gamma}}, \qquad
			c_i \ne 0,
			\quad i = 1,2.
		\]
	Moreover, because they are tangent, we have
	\begin{equation}\label{proof:eq:local4}
		(c\bar{c} + r_1 - r_2)^2 = 4c\bar{c}r_1
		\quad \text{for }
		c := c_2 - c_1
	\end{equation}
	by Lemma~\ref{prop:inters_circles_first_type}. Notice that $c\bar{c} = -c^2$. Let us write $r_2$ as
	\[
		r_2 = \frac{c_2^2}{c_1^2}r_1 = \left( \frac{c}{c_1} + 1\right)^2 r_1.
	\]
	Equation~\eqref{proof:eq:local4} now reads
	\[
		\left(
			\left(
				\frac{c}{c_1} + 1
			\right)^2 r_1 - r_1 + c^2
		\right)^2
		= -4c^2 r_1,
	\]
	or, equivalently,
	\begin{equation}\label{proof:eq:local5}
		\left(
			c \frac{r_1+c_1^2}{c_1^2} + \frac{2r_1}{c_1}
		\right)^2
		= -4 r_1,
	\end{equation}
	where we used that $c\ne 0$ (because $B^1_{(c_1,r_1)}$ and $B^1_{(c_2,r_2)}$ are different). 
	We have a closer look at equation~\eqref{proof:eq:local5}. For this, define
	\[
		\iota := c\  \frac{r_1+c_1^2}{c_1^2} + \frac{2r_1}{c_1}.
	\]
	Observe that $\bar{\iota} = -\iota$, which means that $\iota$ is on the circle 
	$z+\bar{z} = 0$. This implies that in order for~\eqref{proof:eq:local5} to be 
	solvable, we need the square root of $-4r_1$ to be on that circle as well. 
	Since $-4r_1\in GF(q)$, the square root always exists in $GF(q^2)$, 
	and we conclude that $-r_1$ must be a 
	nonsquare in $GF(q)$. If we write $\sqrt{-r_1}$ as
	\[
		\sqrt{-r_1} = c_1\frac{\gamma - \bar{\gamma}}{
			2\sqrt{-\gamma\bar{\gamma}}},
	\]
	it becomes clear that $\bar{\sqrt{-r_1}} = -\sqrt{-r_1}$ if and 
	only if $-\gamma\bar{\gamma}$ is a nonsquare in $GF(q)$. In this case, 
	we can solve equation~\eqref{proof:eq:local5} for $c$:
	\begin{equation}\label{proof:eq:local6}
		c = \frac{c_1^2}{r_1+c_1^2} \left(
			\pm 2\sqrt{-r_1} - \frac{2r_1}{c_1}
		\right).
	\end{equation}
	We should also mention here that $r_1 + c_1^2 \ne 0$, i.e.\ 
	$\frac{(\gamma - \bar{\gamma})^2}{4\gamma\bar{\gamma}}\ne -1$. 
	This follows from the condition $\gamma^2\ne \bar{\gamma}^2$, because
	\[
		(\gamma - \bar{\gamma})^2 = -4\gamma\bar{\gamma}
		\iff \gamma^2 + 2 \gamma\bar{\gamma} + \bar{\gamma}^2 = 0
		\iff (\gamma + \bar{\gamma})^2 = 0.
	\]
	We further simplify \eqref{proof:eq:local6} by using the relation $c_2 = c + c_1$:
	\begin{align*}
		c_2 &= \frac{\pm 2c_1^2\sqrt{-r_1} - 2c_1 r_1}{r_1 + c_1^2} + c_1 
		= \frac{-2c_1\sqrt{-r_1} (\pm c_1 + \sqrt{-r_1})}{
			(\sqrt{-r_1} - c_1)(\sqrt{-r_1} + c_1)} + c_1 \\[0.7ex]
		&= \frac{-2c_1\sqrt{-r_1} + c_1(\sqrt{-r_1} \mp c_1)}{
			\sqrt{-r_1} \mp c_1}
		= c_1 \frac{\mp c_1 - \sqrt{-r_1}}{\mp c_1 + \sqrt{-r_1}} \\[0.7ex]
		&= c_1 \frac{c_1 \pm \sqrt{-r_1}}{c_1 \mp \sqrt{-r_1}}.
	\end{align*}
	We conclude the proof by plugging in the term for $\sqrt{-r_1}$:
	\[
		c_2 = c_1 \, \frac{1 \pm 
			\frac{\gamma - \bar{\gamma}}{2\sqrt{-\gamma\bar{\gamma}}}
			}{1 \mp
			\frac{\gamma - \bar{\gamma}}{2\sqrt{-\gamma\bar{\gamma}}}
			}
		= c_1\, \frac{
			2\sqrt{-\gamma\bar{\gamma}} \pm (\gamma - \bar{\gamma})}{
			2\sqrt{-\gamma\bar{\gamma}} \mp (\gamma - \bar{\gamma})}.
	\]
\end{proof}
\begin{lem}\label{lem:second_circle_3}
	Let $B^1_{(c_1,r_1)},B^1_{(c_2,r_2)}\in \tau(B^2_{(\gamma,0)},B^2_{(\bar{\gamma},0)})$ with
	\[
	c_1 = \bar{c}_1
	\quad \text{and} \quad
	c_2 = -\bar{c}_2.
	\]
	The circles $B^1_{(c_1,r_1)}$ and $B^1_{(c_2,r_2)}$ are tangent if and only if
	\[
		c_2 = \pm c_1\cdot \frac{\gamma - \bar{\gamma}}{
			\gamma + \bar{\gamma}}.
	\]
\end{lem}
\begin{proof}
	By Lemma~\ref{lem:condition_symmetric_carrier_circles} we have
	\[
		r_1 = c_1^2\,\frac{(\gamma + \bar{\gamma})^2}{
			4\gamma\bar{\gamma}}
		\quad \text{and} \quad
		r_2 = c_2^2\,\frac{(\gamma - \bar{\gamma})^2}{
			4\gamma\bar{\gamma}}.
	\]
	We can write $r_2$ as
	\[
		r_2 = c_2^2\left(
			\frac{(\gamma + \bar{\gamma})^2}{4\gamma\bar{\gamma}}
			- \frac{4\gamma\bar{\gamma}}{4\gamma\bar{\gamma}}
		\right)
		= c_2^2 \left(
			\frac{r_1}{c_1^2} - 1
		\right).
	\]
	Furthermore, for $c:= c_2 - c_1$ we have
	\[
		c\bar{c} = (c_2 - c_1)(-c_2 - c_1) = c_1^2 - c_2^2.
	\]
	We use these observations to transform the equation $(c\bar{c} + r_1 - r_2)^2 = 4c\bar{c}r_1$ for two tangent circles of the first type (see Lemma~\ref{prop:inters_circles_first_type}). We find that
	\begin{align*}
		(c\bar{c} + r_1 - r_2)^2 - 4c\bar{c}r_1
		&= \left(
			c_2^2\left( \frac{r_1}{c_1^2} - 1 \right)
			- r_1 + c_2^2 - c_1^2 \right)^2
			- 4(c_1^2 - c_2^2)r_1 \\
		&= \left(
			\left(\frac{c_2^2}{c_1^2} - 1\right)r_1 - c_1^2 \right)^2
			+ 4(c_2^2 - c_1^2)r_1 \\
		&= \left(\frac{c_2^2}{c_1^2} - 1 \right)^2 r_1^2
			+ 2r_1(c_2^2 - c_1^2) + c_1^4 \\
		&= \left(
			\left(\frac{c_2^2}{c_1^2} - 1\right)r_1 + c_1^2
			\right)^2,
	\end{align*}
	where the last term is zero if and only if
	\[
		(c_2^2 - c_1^2)r_1 + c_1^4 = 0,
	\]
	which is equivalent to
	\[
		c_2^2 = c_1^2 \left(1 - \frac{c_1^2}{r_1} \right).
	\]
	The desired result now follows from the fact that
	\[
		1-\frac{c_1^2}{r_1}
		= 1 - \frac{4\gamma\bar{\gamma}}{(\gamma + \bar{\gamma})^2}
		= \frac{(\gamma - \bar{\gamma})^2}{(\gamma + \bar{\gamma})^2}.
	\]
\end{proof}
Let us make a few comments about what we just proved in Lemmas~\ref{lem:second_circle_1}--\ref{lem:second_circle_3}:
\begin{itemize}
	\item
		The case where $c_1 = -\bar{c}_1$ and $c_2 = \bar{c}_2$ can immediately be derived from Lemma~\ref{lem:second_circle_3} by interchanging $c_1$ and $c_2$.
	\item
		In all three lemmas, the condition allows for exactly two circles $B^1_{(c_2,r_2)}$ tangent to $B^1_{(c_1,r_1)}$.
	\item
		In Lemma~\ref{lem:second_circle_1} we obtain $c_2$ from $c_1$ by multiplying $c_1$ with an element $u\in GF(q)$ (which the reader may easily verify by calculating the conjugate of $u$). The same is true for Lemma~\ref{lem:second_circle_2}.
	\item
		It does not matter which square root we choose for $\gamma\bar{\gamma}$ or for $-\gamma\bar{\gamma}$; the equations in Lemma~\ref{lem:second_circle_1} and Lemma~\ref{lem:second_circle_2} stay the same.
	\item
		The radii of $B^1_{(c_1,r_1)}$ and $B^1_{(c_2,r_2)}$ in each case are uniquely 
		determined by $c_1$ and $c_2$, respectively (see Lemma~\ref{lem:condition_symmetric_carrier_circles}).
\end{itemize}
The following corollary is an important observation about the 
restriction on $\gamma$ as given in Lemmas~\ref{lem:second_circle_1} and~\ref{lem:second_circle_2}.
\begin{cor}\label{cor:gamma_square}
	\begin{enumerate}
		\item \label{cor:itm:square_equivalence_1}
		$\gamma\bar{\gamma}$ is a square in $GF(q)$ if and only if $\gamma$ is a square in $GF(q^2)$.
		
		\item \label{cor:itm:square_equivalence_2}
		$-\gamma\bar{\gamma}$ is a nonsquare in $GF(q)$ if and only if either
		\begin{itemize}
			\item
			$\gamma$ is a square in $GF(q^2)$ and $-1$ is a nonsquare in $GF(q)$, or
			\item
			$\gamma$ is a nonsquare in $GF(q^2)$ and $-1$ is a square in $GF(q)$.
		\end{itemize}
	\end{enumerate}
\end{cor}
\begin{proof}
	Recall that an element $b\in GF(q)\setminus\{0\}$ is a square in $GF(q)$ if and 
	only if $b^{\frac{q-1}2}=1$. Hence, by
		\[
			\left(\gamma\bar{\gamma}\right)^{\frac{q-1}{2}}
			= \left(\gamma^{q+1}\right) ^{\frac{q-1}{2}}
			= 1 \iff
			\gamma^{\frac{q^2-1}{2}} = 1,
		\]
		it follows that $\gamma\bar{\gamma}$ is a square in $GF(q)$ if and 
		only if $\gamma$ is a square in $GF(q^2)$, which proves~\ref{cor:itm:square_equivalence_1}.
		
\ref{cor:itm:square_equivalence_2} follows easily from the Facts~\ref{facts}.
\end{proof}
Summarizing, we have established that every circle $B^1_{(c_1,r_1)}\in \tau(B^2_{(\gamma,0)},B^2_{(\bar{\gamma},0)})$ 
has -- under the right circumstances -- four tangent circles in $\tau(B^2_{(\gamma,0)},B^2_{(\bar{\gamma},0)})$.

We will now show that a proper Steiner chain (in accordance with 
Definition~\ref{def:steiner-chain}) can only be constructed in the 
case of Lemma~\ref{lem:second_circle_1} or~\ref{lem:second_circle_2}. 
If $c_1 = \bar{c}_1$ and $c_2 = -\bar{c}_2$ (or vice versa), the 
contact point of $B^1_{(c_1,r_1)}$ and $B^1_{(c_2,r_2)}$ lies on one 
of the carrier circles, which is a violation of Definition~\ref{def:steiner-chain}\ref{def:itm:contact-points}.

To see this, we consult Lemma~\ref{prop:inters_circles_first_type}, where it follows that 
$B^1_{(c_1,r_1)}$ touches $B^2_{(\gamma,0)}$ at
\[
	\zeta^{(1)}_{\gamma}
		= \frac{c_1\bar{\gamma} - \bar{c}_1\gamma}{2\bar{\gamma}}
		= c_1 \frac{\bar{\gamma} - \gamma}{2\bar{\gamma}}
\]
and $B^2_{(\bar{\gamma},0)}$ at
\[
	\zeta^{(1)}_{\bar{\gamma}}
		= \frac{c_1\gamma - \bar{c}_1\bar{\gamma}}{2\gamma}
		= c_1 \frac{\gamma - \bar{\gamma}}{2\gamma}.
\]
Recall that for $B^1_{(c_2,r_2)}$ as given in Lemma~\ref{lem:second_circle_3} we have
\begin{equation}\label{eq:c_case_second_circle_3}
	c_2 = -\bar{c}_2 = \pm c_1 \frac{\gamma - \bar{\gamma}}{\gamma + \bar{\gamma}}.
\end{equation}
Consequently, $B^1_{(c_2,r_2)}$ has the point
\[
	\zeta^{(2)}_{\gamma}
		= \frac{c_2\bar{\gamma} - \bar{c}_2\gamma}{2\bar{\gamma}}
		= c_2 \frac{\bar{\gamma} + \gamma}{2\bar{\gamma}}
		= \pm c_1 \frac{\gamma - \bar{\gamma}}{\gamma + \bar{\gamma}}
			\cdot \frac{\bar{\gamma} + \gamma}{2\bar{\gamma}}
		= \pm c_1 \frac{\gamma - \bar{\gamma}}{2\bar{\gamma}}
\]
in common with $B^2_{(\gamma,0)}$, whereas it shares the point
\[
	\zeta^{(2)}_{\bar{\gamma}}
		= \frac{c_2\gamma - \bar{c}_2\bar{\gamma}}{2\gamma}
		= c_2 \frac{\gamma + \bar{\gamma}}{2\gamma}
		= \pm c_1 \frac{\gamma - \bar{\gamma}}{\gamma + \bar{\gamma}}
			\cdot \frac{\gamma + \bar{\gamma}}{2\gamma}
		= \pm c_1 \frac{\gamma - \bar{\gamma}}{2\gamma}
\]
with $B^2_{(\bar{\gamma},0)}$.

Depending on the sign we choose in~\eqref{eq:c_case_second_circle_3}, we 
find that either $\zeta^{(2)}_{\gamma}$ corresponds to $\zeta^{(1)}_{\gamma}$, 
or $\zeta^{(2)}_{\bar{\gamma}}$ to $\zeta^{(1)}_{\bar{\gamma}}$. 
In either case, we find a point that is contact point of three tangent circles.

Similarly, it is easy to verify that if both $c_1$ and $c_2$ are in 
$z-\bar{z} = 0$ (Lemma~\ref{lem:second_circle_1}) or in $z+\bar{z} = 0$ 
(Lemma~\ref{lem:second_circle_2}), there are no points shared by more than two tangent circles.

To summarize, we can conclude that if $\tau(B^2_{(\gamma,0)},B^2_{(\bar{\gamma},0)})$ 
is non-empty, any circle in $\tau(B^2_{(\gamma,0)},B^2_{(\bar{\gamma},0)})$ 
has exactly two tangent circles which would potentially allow the construction 
of a Steiner chain. In other words, if we can find a Steiner chain starting 
from a given circle, the chain is unique.

According to our earlier reflections, we have to consider two separate cases. 
We start with the case where $B^1_{(c_1,r_1)}$ and $B^1_{(c_2,r_2)}$ are given as in Lemma~\ref{lem:second_circle_1}.

\subsubsection{Case $c_1 = \bar{c}_1$ and $c_2 = \bar{c}_2$}

Let us assume that $\gamma\bar{\gamma}$ is a square in $GF(q)$. We have seen (Corollary~\ref{cor:gamma_square}) that this is equivalent to $\gamma$ being a square in $GF(q^2)$. Moreover, we mentioned earlier that $\gamma$ is a square if and only if $\bar{\gamma}$ is a square.
Therefore, we can write equation~\eqref{lem:eq:second_circle_1} from Lemma~\ref{lem:second_circle_1} as
\begin{equation}\label{eq:local1}
	c_2 = c_1\cdot \frac{2\sqrt{\gamma}\sqrt{\bar{\gamma}}
		\pm (\gamma + \bar{\gamma})}{
		2\sqrt{\gamma}\sqrt{\bar{\gamma}}
		\mp (\gamma + \bar{\gamma})}.
\end{equation}
Define
\[
	u_1 := \sqrt{\gamma} + \sqrt{\bar{\gamma}},
	\qquad
	u_2 := \sqrt{\gamma} - \sqrt{\bar{\gamma}},
\]
and
\[
	u := - \left(\frac{u_1}{u_2}\right)^2.
\]
Then the two possibilities in~\eqref{eq:local1} correspond to
\[
	c_2 = u\cdot c_1
	\quad \text{and} \quad
	c_2 = \frac{1}{u}\cdot c_1.
\]
As we saw in earlier calculations, $u$ is in  $GF(q)\setminus\{0\}$. 
Let $k$ be the multiplicative order of $u$ in $GF(q)\setminus\{0\}$, 
i.e.\ $u^k = 1$ but $u^l\ne 1$ for $1<l<k$. We need to note a few 
observations regarding the multiplicative order $\ord(u)$ of $u$:

\begin{rem}
	\begin{itemize}
		\item
		The multiplicative order of $u$ in $GF(q)\setminus\{0\}$ is the same as 
		its multiplicative order in $GF(q^2)\setminus\{0\}$, or in any other 
		extension field for that matter. Thus, we will henceforth not 
		specify which cyclic group we refer to if we talk about the multiplicative order of $u$.
		
		\item
		$\ord(u) = \ord(\frac{1}{u})$.
		
		\item
		$\ord(u) > 1$, or, in other words, $u\ne 1$. This follows with equation~\eqref{eq:local1} from the fact that $\gamma^2\ne \bar{\gamma}^2$.
		
		\item
		$\ord(u) \mid q-1$, since the order of any element divides the order of the group.
	\end{itemize}
\end{rem}

Apparently, if $\ord(u) = k$ and $c_1$ is any element of $GF(q)\setminus\{0\}$, the chain of circles
\[
	B^1_{(c_1, r_1)} \to
	B^1_{(u c_1, r_2)} \to
	B^1_{(u^2 c_1, r_3)} \to \cdots \to
	B^1_{(u^k c_1, r_{k+1})} = B^1_{(c_1, r_1)}
\]
with
\[
	r_i := (u^{i-1}c_1)^2
		\frac{(\gamma + \bar{\gamma})^2}{4\gamma\bar{\gamma}}
\]
defined as in Lemma~\ref{lem:condition_symmetric_carrier_circles}, is a 
Steiner chain of length $k$. In fact, we can build such a chain 
starting with any element $c_1$ of $GF(q)\setminus\{0\}$. Consequently, 
if $\gamma$ is a square in $GF(q^2)$, there are 
$\frac{q-1}{k}$ Steiner chains, and each chain has length $k$.

Since the length of the Steiner chains depends on the multiplicative order of $u$, we have a closer look at $u$. If we write $\frac{u_1}{u_2}$ as
\[
	\frac{u_1}{u_2} = \frac{\sqrt{\gamma} + \sqrt{\bar{\gamma}}}{
		\sqrt{\gamma} - \sqrt{\bar{\gamma}}}
	= \frac{(\sqrt{\gamma} + \sqrt{\bar{\gamma}})^2}{
		(\sqrt{\gamma} - \sqrt{\bar{\gamma}})
		(\sqrt{\gamma} + \sqrt{\bar{\gamma}})
	}
	= \frac{\gamma + \bar{\gamma} + 2\sqrt{\gamma{\bar{\gamma}}}}{
		\gamma - \bar{\gamma}},
\]
it is easy to see that $\bar{\frac{u_1}{u_2}} = -\frac{u_1}{u_2}$, i.e.\ $\left(\frac{u_1}{u_2}\right)^2$ is a nonsquare in $GF(q)$. We know that $u = -1\cdot \left(\frac{u_1}{u_2}\right)^2$, and hence we
have to distinguish between two cases:
\begin{itemize}
	\item
	If $-1$ is a square in $GF(q)$, then $u$ is a nonsquare in $GF(q)$. In this case, 
	the multiplicative order of $u$ is a divisor of $q-1$, but does not divide $\frac{q-1}{2}$.
	\item
	If $-1$ is a nonsquare in $GF(q)$, $u$ is a square in $GF(q)$, and  
	the multiplicative order of $u$ divides $\frac{q-1}{2}$.
	
	Notice that if $-1$ is a nonsquare in $GF(q)$, $m$ is odd and $p\equiv 3\mod 4$. If we write $\bar{p} = \bar{3}\in \ZZ_4$ and $m = 2d + 1$, it follows that
	\[
		\bar{p}^m = (\bar{3}^2)^d \cdot \bar{3}
		= \bar{1}^d \cdot \bar{3} = \bar{3}.
	\]
	Consequently, $\frac{q-1}{2}$ is not divisible by $2$, and therefore, the length of the Steiner chain is odd.
\end{itemize}

\subsubsection{Case $c_1 = -\bar{c}_1$ and $c_2 = -\bar{c}_2$}

We assume that $-\gamma\bar{\gamma}$ is a nonsquare in $GF(q)$ as required by Lemma~\ref{lem:second_circle_2}. Recall equation~\eqref{lem:eq:second_circle_2} in said Lemma:
\begin{equation}\label{eq:local2}
	c_2 = c_1\cdot \frac{
		2\sqrt{-\gamma\bar{\gamma}} \pm (\gamma - \bar{\gamma})}{
		2\sqrt{-\gamma\bar{\gamma}} \mp (\gamma - \bar{\gamma})}.
\end{equation}
Define
\[
	v_1 := \gamma + \sqrt{-1}\sqrt{\gamma\bar{\gamma}},
	\qquad
	v_2 := \sqrt{-1}\gamma + \sqrt{\gamma\bar{\gamma}},
\]
and
\[
	v := \left(\frac{v_1}{v_2}\right)^2.
\]
The reader may verify that the two possibilities in~\eqref{eq:local2} correspond to
\[
	c_2 = v\cdot c_1
	\quad \text{and} \quad
	c_2 = \frac{1}{v}\cdot c_1.
\]
With equation~\eqref{eq:local2} it is easy to see that $v\in GF(q)\setminus\{0\}$ and $v\ne 1$. We denote by $k'$ the multiplicative order of $v$ in $GF(q)\setminus\{0\}$ and let $c_1$ be any of the $q-1$ elements in $B^{2}_{(1,0)}\setminus \{0,\infty\}$. A Steiner chain of length $k'$ is then given by
\[
	B^1_{(c_1, r_1)} \to
	B^1_{(v c_1, r_2)} \to
	B^1_{(v^2 c_1, r_3)} \to \cdots \to
	B^1_{(v^{k'} c_1, r_{k'+1})} = B^1_{(c_1, r_1)}
\]
wih $r_i$ determined by Lemma~\ref{lem:condition_symmetric_carrier_circles}:
\[
	r_i := (v^{i-1}c_1)^2
		\frac{(\gamma - \bar{\gamma})^2}{4\gamma\bar{\gamma}}.
\]
We can construct such a chain for any element $c_1\ne 0$ in $z+\bar{z}=0$, which means that there are $\frac{q-1}{k'}$ possible Steiner chains.

The length of the Steiner chains depends on the multiplicative order of $v$. Let us therefore have a closer look at $v$. We notice that a square root of $v$ is given by
\begin{align}\label{eq:local3}
	\sqrt{v} &= \frac{v_1}{v_2}
	= \frac{(v_1)^2}{v_1v_2}
	= \frac{(\gamma + \sqrt{-1}\sqrt{\gamma\bar{\gamma}})^2}{
		\sqrt{-1}(\gamma^2 + \gamma\bar{\gamma})}
	= \frac{2\sqrt{-1}\sqrt{\gamma\bar{\gamma}} + \gamma -\bar{\gamma}}{
		\sqrt{-1}(\gamma + \bar{\gamma})} \nonumber\\
	&= \frac{2\sqrt{\gamma\bar{\gamma}} + \sqrt{-1}(\bar{\gamma} - \gamma)}{
		\gamma + \bar{\gamma}}.
\end{align}
By assumption, $-\gamma\bar{\gamma}$ is a nonsquare in $GF(q)$, which means that exactly one of $-1$ and $\gamma\bar{\gamma}$ is a square in $GF(q)$, see Corollary~\ref{cor:gamma_square}. By~\eqref{eq:local3}, we can say that if $-1$ is a square in $GF(q)$, then $\bar{\sqrt{v}} = -\sqrt{v}$, and otherwise, $\bar{\sqrt{v}} = \sqrt{v}$.

Accordingly, there are two cases (see also Corollary~\ref{cor:gamma_square}):
\begin{itemize}
	\item
	If $-1$ is a square in $GF(q)$ and $\gamma$ a nonsquare in $GF(q^2)$, then $v$ is a nonsquare in $GF(q)$. In this case, the multiplicative order of $v$ is a divisor of $q-1$, but does not divide $\frac{q-1}{2}$.
	\item
	If $-1$ is a nonsquare in $GF(q)$ and $\gamma$ a square in $GF(q^2)$, then $v$ is a square in $GF(q)$, and the multiplicative order of $v$ divides $\frac{q-1}{2}$. By above reasoning, the length of any Steiner chain in this case is always odd.
\end{itemize}

\subsubsection{Overview}
Let us summarize what we have shown so far. Remember that $-1$ is a nonsquare in $GF(q)$ 
if and only if $q\equiv 3 \mod 4$.

\begin{thm}\label{thm:4cases}
	Let $B^2_{(\gamma,0)}$ and $B^2_{(\bar{\gamma},0)}$ be two different 
	circles of the second type (i.e.\ $\gamma^2\ne \bar{\gamma}^2$). Define
	\[
		u := \frac{
			2\sqrt{\gamma\bar{\gamma}} + (\gamma + \bar{\gamma})}{
			2\sqrt{\gamma\bar{\gamma}} - (\gamma + \bar{\gamma})}
		\quad \text{and} \quad
		v := \frac{
			2\sqrt{-\gamma\bar{\gamma}} + (\gamma - \bar{\gamma})}{
			2\sqrt{-\gamma\bar{\gamma}} - (\gamma - \bar{\gamma})},
	\]
	and let $k$ and $k'$ be the multiplicative orders of $u$ and $v$, respectively.
	\begin{enumerate}
		\item If $-1$ is a nonsquare in $GF(q)$ and\label{i}
		\begin{enumerate}
			\item
				$\gamma$ is a square in $GF(q^2)$, there are $\frac{q-1}{k}$ Steiner chains of length $k$ and $\frac{q-1}{k'}$ Steiner chains of length $k'$.\label{ia}
			\item
				$\gamma$ is a nonsquare in $GF(q^2)$, there are no Steiner chains.\label{ib}
		\end{enumerate}
		
		\item If $-1$ is a square in $GF(q)$ and
		\begin{enumerate}
			\item\label{thm:itm:ss}
				$\gamma$ is a square in $GF(q^2)$, there are $\frac{q-1}{k}$ Steiner chains of length $k$ each.
			\item\label{thm:itm:ns}
				$\gamma$ is a nonsquare in $GF(q^2)$, there are $\frac{q-1}{k'}$ Steiner chains of length $k'$.
		\end{enumerate}
	\end{enumerate}
	
	In~\ref{ia} the length of every Steiner chain is odd and a divisor of $\frac{q-1}{2}$.
	In~\ref{thm:itm:ss} and~\ref{thm:itm:ns} the length of the Steiner chains does not divide $\frac{q-1}{2}$.
\end{thm}

Notice that if $-1$ is a square in $GF(q)$, Steiner chains always exist, and exactly $q-1$ circles are part of a Steiner chain. If $-1$ is a nonsquare in $GF(q)$ and $\gamma$ a square, then there are $2(q-1)$ circles used in Steiner chains.

\subsection{The general case}
\label{sec:inters_circles_general}
Let $C_1\neq C_2$ be two arbitrary circles 
with two intersection points $z_1$ and $z_2$.
A M\"obius transformation $T$ which 
maps $z_1$ to $0$ and $z_2$ to $\infty$, maps
$C_1$ and $C_2$ to two circles of the second type
$B^2_{(\gamma_1,0)}$ and $B^2_{(\gamma_2,0)}$. Since
$C_1$ and $C_2$ are different, we have 
$\gamma_1\bar{\gamma}_2 - \bar{\gamma}_1\gamma_2\ne 0$. 
And $C_1$ and $C_2$ carry a Steiner chain if and only if 
$B^2_{(\gamma_1,0)}$ and $B^2_{(\gamma_2,0)}$ carry a Steiner chain.

We observed that $\gamma_1\gamma_2$ must be a square in $GF(q^2)$ 
in order for a Steiner chain to exist, and we showed that this is 
the case if and only if $\frac{\bar{\gamma}_2}{\bar{\gamma}_1}$ is a square. 
Hence, if this condition is satisfied, we were able to map the circles 
$B^2_{(\gamma_1,0)}$ and $B^2_{(\gamma_2,0)}$ to $B^2_{(\gamma,0)}$ 
and $B^2_{(\bar{\gamma},0)}$, where
\[
	\gamma := \sqrt{\frac{\bar{\gamma}_2}{\bar{\gamma}_1}},
\]

and the condition $\gamma_1\bar{\gamma}_2 \ne \bar{\gamma}_1\gamma_2$ changes to $\gamma^2\ne \bar{\gamma}^2$.

But what does this mean for two arbitrary intersecting circles? What is the necessary condition for two arbitrary intersecting circles $C_1,C_2$ to carry a Steiner chain?
This is where the capacitance comes in (see Section~\ref{sec-preliminaries}). 
Recall that the capacitance of $B^2_{(\gamma_1,0)}$ and $B^2_{(\gamma_2,0)}$ is defined as
\[
	\capac = \frac{1}{\gamma_1\bar{\gamma}_1\gamma_2\bar{\gamma}_2}(
	\gamma_1\bar{\gamma}_2 + \bar{\gamma}_1\gamma_2)^2
	= \frac{\bar{\gamma}_2}{\bar{\gamma}_1}
	\cdot \frac{\gamma_1}{\gamma_2} + 2
	+ \frac{\bar{\gamma}_1}{\bar{\gamma}_2}
	\cdot \frac{\gamma_2}{\gamma_1}.
\]
The capacitance of any pair of circles that can be mapped to 
$B^2_{(\gamma_1,0)}$ and $B^2_{(\gamma_2,0)}$ via a M\"obius transformation 
has the same value. Hence, if instead of giving a condition for 
$\frac{\bar{\gamma}_2}{\bar{\gamma}_1}$ we can state a condition for $\capac$, we will be able to decide for two arbitrary intersecting circles whether they may possibly carry a Steiner chain or not by looking at their capacitance. This is the motivation behind the following Lemma.

\begin{lem}\label{lem:capac_first_cond}
	$\frac{\bar{\gamma}_2}{\bar{\gamma}_1}$ is a square in $GF(q^2)$ if and only if either
	\begin{itemize}
		\item
		$\capac = 0$ and $-1$ is a nonsquare in $GF(q)$, or
		\item
		$\capac \ne 0$ is a square in $GF(q)$.
	\end{itemize}
\end{lem}

\begin{proof}
	We substitute $\frac{\bar{\gamma}_2}{\bar{\gamma}_1}$ by $g$ and write $\capac$ as
	\begin{equation}\label{proof:eq:local7}
		\capac = \frac{g}{\bar{g}} + 2 + \frac{\bar{g}}{g}
		= \frac{g^2 + 2g\bar{g} + \bar{g}^2}{g\bar{g}}
		= \frac{(g+\bar{g})^2}{g\bar{g}}.
	\end{equation}
	Since $\capac$ is in $GF(q)$, its square root in $GF(q^2)$ always exists. In particular, it is clear from~\eqref{proof:eq:local7} that if $\capac\ne0$, its square root is in $GF(q)$ if and only if $g\bar{g}$ is a square in $GF(q)$. Having a look at Corollary~\ref{cor:gamma_square}, it is evident that this is equivalent 
	to $g=\frac{\bar{\gamma}_2}{\bar{\gamma}_1}$ being a square in $GF(q^2)$.
	
	On the other hand, if $\capac = 0$, we have $\bar{g} = -g$, which is equivalent to
	\[
		g\bar{g} = -g^2
	\]
	(recall that $g\ne 0$ because of the restriction $\gamma_1\bar{\gamma}_2 - \bar{\gamma}_1\gamma_2 \ne 0$). It follows that a square root of $g\bar{g}$ is given by $\sqrt{-1}g$, and therefore $g\bar{g}$ is a square in $GF(q)$ if and only if $\sqrt{-1}g\in GF(q)$. Since $g = -\bar{g}$, this is the same as requiring that $\sqrt{-1}$ is a nonsquare in 
$GF(q)$ . With Corollary~\ref{cor:gamma_square} we conclude that $g$ is a square in $GF(q^2)$ if and only if $-1$ is a nonsquare in $GF(q)$.
\end{proof}

From now on, let us assume that $\frac{\bar{\gamma}_2}{\bar{\gamma}_1}$ is a square in $GF(q^2)$. In this case we can write
\[
	\capac = \frac{\gamma^2}{\bar{\gamma}^2} + 2 + \frac{\bar{\gamma}^2}{\gamma^2}
	= \left(
		\frac{\gamma}{\bar{\gamma}} + \frac{\bar{\gamma}}{\gamma}
	\right)^2
\]
with $\gamma=\sqrt{\frac{\bar{\gamma}_2}{\bar{\gamma}_1}}$ a square root of $\frac{\bar{\gamma}_2}{\bar{\gamma}_1}$. Notice that $\capac$ (and also the square root of $\capac$) does not depend on which square root of $\frac{\bar{\gamma}_2}{\bar{\gamma}_1}$ we assign to $\gamma$.

At this point Theorem~\ref{thm:4cases} comes into play: We saw that the existence and length of a Steiner chain depends directly on whether $\gamma$ is a square in $GF(q^2)$ or not. Remember that our goal is to prove or disprove the existence of Steiner chains on the basis of the capacitance. To investigate this, we need to find a correlation between $\capac$ and $\gamma$ being a square or a nonsquare.
We consider two separate cases (compare with Lemma~\ref{lem:capac_first_cond}):
\begin{enumerate}
	\item
		$\capac=0$ and $-1$ is a nonsquare in $GF(q)$ (Lemma~\ref{lem:capac_second_cond_zero}), and
	\item
		$\capac\ne 0$ is a square in $GF(q)$ (Lemma~\ref{lem:capac_second_cond}).
\end{enumerate}

But before we have a look at how $\capac$ and $\gamma$ are connected, we need another Lemma, which will be essential for the proof of Lemma~\ref{lem:capac_second_cond_zero}.

\begin{lem}\label{lem:always_square}
	Assume that $-1$ is a nonsquare in $GF(q)$. Then $\sqrt{\gamma\bar{\gamma}}$ is a square in $GF(q^2)$.
\end{lem}
\begin{proof}
	We need to show that
	\[
		\left(\gamma\bar{\gamma}\right)^\frac{q^2-1}{4} = 1
	\]
	(notice that $q^2-1 = (q-1)(q+1)$ is always divisible by $4$). We write the left-hand side as
	\[
		\left(\gamma\bar{\gamma}\right)^\frac{q^2-1}{4}
		=\gamma^{(q+1)\cdot\frac{q^2-1}{4}}
		=\gamma^{\frac{q+1}{4} \cdot (q^2-1)}.
	\]
	Since $-1$ is a nonsquare in $GF(q)$, it follows that $q\equiv 3\mod 4$. 
	This means that $q+1$ is divisible by $4$, and hence $\gamma^{\frac{q+1}{4}}$ exists. 
	Moreover, since $GF(q^2)\setminus\{0\}$ is a cyclic group of order $q^2-1$, 
	any element raised to the power $q^2-1$ is equal to~$1$. Consequently,
	\[
		\gamma^{\frac{q+1}{4} \cdot (q^2-1)}
		=\left(\gamma^\frac{q+1}{4}\right)^{q^2-1}=1,
	\]
	which concludes the proof.
\end{proof}

As we go on, it will be helpful to refer back to Theorem~\ref{thm:4cases} from time to time. Also, the reader may want to have a look at Theorem~\ref{thm:final} and Table~\ref{tab:overview_inters_circles_general} already now to see what we are aiming at.

\begin{lem}\label{lem:capac_second_cond_zero}
	If $\capac = 0$ and $-1$ is a nonsquare in $GF(q)$, then $\gamma$ is a square in $GF(q^2)$ if and only if $p\equiv 7\mod 16$.
\end{lem}

\begin{proof}
	The condition $\capac=0$ is equivalent to
	\[
		\frac{\gamma}{\bar{\gamma}} + \frac{\bar{\gamma}}{\gamma} = 0
		\iff \gamma^2 + \bar{\gamma}^2 = 0
		\iff \gamma = \pm \sqrt{-1} \bar{\gamma}.
	\]
	Be aware that $\gamma\notin GF(q)$, and in particular $\gamma\ne 0$, a consequence of the afore-mentioned property $\gamma^2\ne \bar{\gamma}^2$.
	Multiplying both sides of the equation by $\gamma$ leads to
	\[
		\gamma^2 = \sqrt{-1}\cdot \gamma\bar{\gamma},
	\]
	where we omit the $\pm$-sign by using $\sqrt{-1}$ to represent both square roots of $-1$. If we write $\sqrt{-1}$ as $\sqrt{-1}=\frac{\gamma^2}{\gamma\bar{\gamma}}$, it is obvious that a square root of $\sqrt{-1}$ exists. We can therefore write
	\begin{equation}\label{proof:eq:local8}
		\gamma = \sqrt{\sqrt{-1}}\cdot \sqrt{\gamma\bar{\gamma}}.
	\end{equation}
	Again, we omit the $\pm$-sign, as it is irrelevant for our considerations which square root we take.

	Because of Lemma~\ref{lem:always_square} we know that the square root 
	of $\sqrt{\gamma\bar{\gamma}}$ exists. It is now obvious 
	from~\eqref{proof:eq:local8} that $\gamma$ is a square if and only if $\sqrt{\sqrt{-1}}$ is 
	a square. This is the case if and only if the multiplicative order of $\sqrt{\sqrt{-1}}$ is a 
	divisor of $\frac{q^2-1}{2}$. Since $-1$ has multiplicative order $2$, it follows 
	that the multiplicative order of $\sqrt{\sqrt{-1}}$ is $8$. This implies that $\gamma$ is 
	a square in $GF(q^2)$ if and only if $q^2-1$ is divisible by $16$, i.e.\ if and only if $q^2\equiv 1\mod 16$.
	
	What does this mean for $p$ and $m$? Recall that by assumption, $m$ is odd and $p\equiv 3\mod 4$. 
	For $\bar{p}\in \ZZ_{16}$ there are thus four possibilities: $\bar{p} = \bar{3},\ \bar{p} = \bar{7},\ \bar{p} = \overline{11}$, or $\bar{p} = \bar{15}$.
	
	Let us write $m=2d+1$. If $p\equiv 3\mod 16$, we see that
	\[
		\bar{p}^{2m} = \bar{3}^{2(2d+1)} = \bar{9}^{2d+1} = \bar{1}^d\cdot \bar{9} = \bar{9}.
	\]
	Similarly, one checks that $\overline{11}^{2m} = \bar{9}$ and $\overline{15}^{2m} = \bar{9}$. The only case where 
	$q^2\equiv 1\mod 16$ is for $\bar{p} = \bar{7}$:
	\[
		\bar{7}^{2(2d+1)} = \big(\bar{7}^2\big)^{2d+1} = \bar{1}^{2d+1} = \bar{1}.
	\]
\end{proof}
Recall that
$\capac = \left(
	\frac{\gamma}{\bar{\gamma}} + \frac{\bar{\gamma}}{\gamma}
\right)^2$.
In the following Lemma, we define $\sqrt{\capac}$ to be equal to $\frac{\gamma}{\bar{\gamma}} + \frac{\bar{\gamma}}{\gamma}$. Be aware that this is an arbitrary definition. If in general we calculate the square root of the capacitance of two given circles, it is not clear from the outset whether the square root we take corresponds to $\frac{\gamma}{\bar{\gamma}} + \frac{\bar{\gamma}}{\gamma}$ or to $-\frac{\gamma}{\bar{\gamma}} - \frac{\bar{\gamma}}{\gamma}$.

\begin{lem}\label{lem:capac_second_cond}
	Assume that $\capac\ne 0$ is a square in $GF(q)$ with square root $\sqrt{\capac} = \frac{\gamma}{\bar{\gamma}} + \frac{\bar{\gamma}}{\gamma}$. Then:
	\begin{enumerate}
		\item\label{lem:enum:capacitance_equivalence1}
		If $-1$ is a nonsquare in $GF(q)$, the following are equivalent:
		
		$\gamma$ is a square in $GF(q^2) \iff \sqrt{\capac}+2$ is a square in $GF(q) \iff -\sqrt{\capac} + 2$ is a square in $GF(q)$.

		\item \label{lem:enum:capacitance_equivalence2}
		If $-1$ is a square in $GF(q)$, the following are equivalent:
			
		$\gamma$ is a square in $GF(q^2) \iff \sqrt{\capac}+2$ is a square in 
		$GF(q) \iff -\sqrt{\capac}+2$ is a nonsquare in $GF(q)$.
	\end{enumerate}
\end{lem}

\begin{proof}
	We treat the two cases $\sqrt{\capac} + 2$ and $-\sqrt{\capac} + 2$ separately:
	\begin{itemize}
		\item
		For $\sqrt{\capac} + 2$ we have
		\[
			\sqrt{\capac} + 2
			= \frac{\gamma}{\bar{\gamma}} +2 + \frac{\bar{\gamma}}{\gamma}
			= \frac{\gamma^2 + 2\gamma\bar{\gamma} + \bar{\gamma}^2}{
				\gamma\bar{\gamma}}
			= \frac{(\gamma + \bar{\gamma})^2}{\gamma\bar{\gamma}}.
		\]
		Note that $\gamma + \bar{\gamma}\ne 0$ as $\gamma^2\ne \bar{\gamma}^2$.
		Obviously, $\sqrt{\capac} + 2$ is a square in $GF(q)$ if and only if $\gamma\bar{\gamma}$ is a square in $GF(q)$, which is the case if and only if $\gamma$ is a square in $GF(q^2)$ -- see Corollary~\ref{cor:gamma_square}.
			
		\item
		Conversely, for $-\sqrt{\capac} + 2$ we can write
		\[
			- \sqrt{\capac} + 2
			= - \frac{\gamma}{\bar{\gamma}} +2 - \frac{\bar{\gamma}}{\gamma}
			= \frac{\gamma^2 - 2\gamma\bar{\gamma} + \bar{\gamma}^2}{
				-\gamma\bar{\gamma}}
			= \frac{(\gamma - \bar{\gamma})^2}{-\gamma\bar{\gamma}}.
		\]
		Note that $\gamma - \bar{\gamma}\ne 0$ as $\gamma^2\ne \bar{\gamma}^2$.
		Here, $-\sqrt{\capac} + 2$ is a square in $GF(q)$ if and only if $-\gamma\bar{\gamma}$ 
		is a nonsquare in $GF(q)$. Again, the desired result follows with Corollary~\ref{cor:gamma_square}.
	\end{itemize}
\end{proof}

\begin{rem}
	If $\sqrt{\capac}$ is an arbitrary square root of $\capac$ and $-1$ a nonsquare in $GF(q)$, then $\sqrt{\capac} + 2$ is a square in $GF(q)$ if and only if $\gamma$ is a square in $GF(q^2)$ (case~\ref{lem:enum:capacitance_equivalence1} of Lemma~\ref{lem:capac_second_cond}). On the other hand, if $-1$ is a square in $GF(q)$ (case~\ref{lem:enum:capacitance_equivalence2} of Lemma~\ref{lem:capac_second_cond}), exactly one of $\sqrt{\capac}+2$ and $-\sqrt{\capac}+2$ is a square in $GF(q)$. A Steiner chain in this case always exists: we are either in case~\ref{thm:itm:ss} or in case~\ref{thm:itm:ns} of Theorem~\ref{thm:4cases}.
\end{rem}

We are now well on the way to proving our main theorem.
What we still lack is a condition for the length of the Steiner chains in case they exist. For this, let us recall the definitions of $u$ and $v$ in Theorem~\ref{thm:4cases}:
\[
	u := \frac{
		2\sqrt{\gamma\bar{\gamma}} + (\gamma + \bar{\gamma})}{
		2\sqrt{\gamma\bar{\gamma}} - (\gamma + \bar{\gamma})},
	\qquad
	v := \frac{
		2\sqrt{-\gamma\bar{\gamma}} + (\gamma - \bar{\gamma})}{
		2\sqrt{-\gamma\bar{\gamma}} - (\gamma - \bar{\gamma})}.
\]
We write $u$ and $v$ as
\[
	u = \frac{
		2 + \frac{\gamma + \bar{\gamma}}{\sqrt{\gamma\bar{\gamma}}}}{
		2 - \frac{\gamma + \bar{\gamma}}{\sqrt{\gamma\bar{\gamma}}}}
	\quad\text{ and }\quad
	v = \frac{
		2 + \frac{\gamma - \bar{\gamma}}{\sqrt{-\gamma\bar{\gamma}}}}{
		2 - \frac{\gamma - \bar{\gamma})}{\sqrt{-\gamma\bar{\gamma}}}}.
\]
Notice that
\[
	\left(
		\frac{\gamma \pm \bar{\gamma}}{\sqrt{\pm\gamma\bar{\gamma}}}
	\right)^2
	= \pm\frac{\gamma}{\bar{\gamma}} + 2 \pm \frac{\bar{\gamma}}{\gamma}
	= \pm \sqrt{\capac} + 2.
\]
Apparently, $u$ and $v$ (or $\frac{1}{u}$ and $\frac{1}{v}$, depending on which square root of $\pm\sqrt{\capac}+2$ we take) correspond to
\[
	w^{\pm} := \frac{2 + \sqrt{\pm \sqrt{\capac} + 2}}{
		2 - \sqrt{\pm \sqrt{\capac} + 2}}.
\]
In particular, if $\capac = 0$, we have
\[
	w^{\pm} = \frac{2 + \sqrt{2}}{2 - \sqrt{2}} = \frac{(2 + \sqrt{2})^2}{2} = 3+2\sqrt{2}.
\]

Our results from Section~\ref{sec:inters_circles_general} combined with Theorem~\ref{thm:4cases} 
are summarized in the following

\begin{thm}\label{thm:final}
	Let $C_1$ and $C_2$ be two intersecting circles in $\miqplane(q)$, $q=p^m$, for $p$ an odd prime. Let
	\[
		\capac := \capacfunc(C_1,C_2)
	\]
	be the associated capacitance as defined in Definition~\ref{def:capacitance}, 
	and $\sqrt{\capac}$ any square root of $\capac$. If $\sqrt{\capac}\in GF(q)$, we additionally define
	\[
		w^{\pm} := \frac{2 + \sqrt{\pm \sqrt{\capac} + 2}}{
			2 - \sqrt{\pm \sqrt{\capac} + 2}}.
	\]
	Then, the circles $C_1$ and $C_2$ carry a Steiner chain if and only if one of the 
	following three conditions is satisfied:
	\begin{enumerate}
		\item\label{thm:enum:local1}
			$\capac = 0$, $m$ is odd, and $p\equiv 7\mod 16$.
			
			In this case there are $2\frac{q-1}{k}$ Steiner chains, whose length $k$ is given by the multiplicative order of $3+2\sqrt{2}$.
		
		\item\label{thm:enum:local2}
			$\capac\ne 0$, $\sqrt{\capac}\in GF(q)$, $-1$ is a nonsquare in $GF(q)$, and 
			$\sqrt{\capac}+2$ is a square in $GF(q)$.
			
			There are $\frac{q-1}{k^+}$ Steiner chains of length $k^+$ and $\frac{q-1}{k^-}$ Steiner chains of length $k^-$, where $k^+$ and $k^-$ are the multiplicative orders of $w^+$ and $w^-$, respectively.
			
		\item\label{thm:enum:local3}
			$\capac\ne 0$, $\sqrt{\capac}\in GF(q)$, and $-1$ is a square in $GF(q)$.
			
			There are $\frac{q-1}{k}$ Steiner chains of length $k$ each, 
			where $k$ is the multiplicative order of $w^+$ if $\sqrt{\capac}+2$ is a square in $GF(q)$, and the multiplicative order of $w^-$, otherwise.
	\end{enumerate}
	In~\ref{thm:enum:local1} and~\ref{thm:enum:local2}, the length of the chains is odd and a divisor of $\frac{q-1}{2}$, whereas the length of the chains in case~\ref{thm:enum:local3} does not divide $\frac{q-1}{2}$.
\end{thm}


\def\arraystretch{1.5}
\begin{table}[!h]
	\centering
	\caption{Overview of Steiner chains for intersecting carrier circles in $\miqplane(q)$}
	\label{tab:overview_inters_circles_general}
	\small
	\begin{tabular}{| p{1.6cm} | p{3.6cm} | p{3.6cm} | p{3.6cm} |}
		\hline
		{\bfseries Case}
			& \multicolumn{2}{c|}{
				$q\equiv 3\mod 4$}
			& $q\equiv 1\mod 4$
			\\ \hline
		{\bfseries Condition}
			& $\capac = 0$ and $p\equiv 7\mod 16$.
			& $\capac\ne 0$ is a square in $GF(q)$ and $\sqrt{\capac}+2$ is a square in $GF(q)$.
			& $\capac\ne 0$ is a square in $GF(q)$.
			\\ \hline
		{\bfseries Result}
			& There are $2\frac{q-1}{k}$ chains of length $k$.
			& There are $\frac{q-1}{k^+}$ chains of length $k^+$ and $\frac{q-1}{k^-}$ chains of length $k^-$.
			& There are $\frac{q-1}{k}$ chains of length $k$.
			\\ \hline
		{\bfseries Comment}
			& $k$ is the multiplicative order of $3+2\sqrt{2}$.
			& $k^+$ is the multiplicative order of $w^+$ and $k^-$ is the multiplicative order of $w^-$.
			& $k$ is the multiplicative order of $w^{\pm}$, where the sign is chosen such that $\pm\sqrt{\capac}+2$ is a square in $GF(q)$.
			\\ \hline
		\hspace{0pt}{\bfseries Specifics}
			& \multicolumn{2}{p{6.4cm}|}{
				The length of the chains is odd and divides $\frac{q-1}{2}$.}
			& The length of the chains divides $q-1$ but does not divide $\frac{q-1}{2}$.
			\\ \hline
		\end{tabular}
\end{table}

\begin{exa}
If $\miqplane(31)$ is constructed over the pair of finite 
fields $GF(31)$ and $GF(31)(\alpha)$ with $\alpha=\sqrt{-1}$, one 
can verify by Lemma~\ref{prop:inters_circles_first_type} 
that the circles $B^1_{(3\alpha+8,14)}$ and $B^2_{(5\alpha+12,17)}$ are intersecting, and 
we compute that their capacitance $\capac$ equals $2$.

A square root of $\capac$ is given by $\sqrt{\capac} = 8$. Moreover, we can determine the following square roots:
\[
	\sqrt{\sqrt{\capac} + 2} = 14
	\quad\text{and}\quad
	\sqrt{-\sqrt{\capac} + 2} = 5.
\]
Obviously, all the requirements for the existence of a Steiner 
chain as stated in Theorem~\ref{thm:final}~\ref{thm:enum:local2} are satisfied.
To determine the length of the Steiner chains, we have a look at $w^{\pm}$:
\[
	w^+ = \frac{2+14}{2-14} = \frac{16}{19}=9
	\qquad
	w^- = \frac{2+5}{2-5} = \frac{7}{28}=8.
\]

The multiplicative orders of $w^+ = 9$ and $w^- = 8$ are
15 and 5, respectively.
Accordingly, $B^1_{(3\alpha+8,14)}$ and $B^2_{(5\alpha+12,17)}$ carry $2$ Steiner chains of length $15$ and $6$ Steiner chains of length $5$. This can be confirmed by an exhaustive search
of circles, implemented in {\scriptsize\sagelogo}. Explicit code can be found in~\cite{villiger}.
\end{exa}

\bibliographystyle{plain}

\end{document}